\documentclass[10pt,a4paper]{amsart}
\usepackage{lmodern}

\usepackage[marginpar=2cm,asymmetric,ignoremp,margin=2.5cm]{geometry}

\usepackage{comment}
\usepackage[bookmarksdepth=3]{hyperref}
\usepackage{doi}
\usepackage{bbm}
\usepackage{cancel}

\allowdisplaybreaks

\def\equationautorefname~#1\null{Equation~(#1)\null}

\usepackage{mathdots}

\renewcommand{\vec}[1]{\mathbf{\uline{#1}}}

\DeclareMathOperator{\id}{id}

\DeclareMathOperator{\bl}{bl}

\newcommand{\dd}{\mathrm{d}}

\newcommand{\ttw}{\mathop{\widetilde{t}}{}}
\newcommand{\ztw}{\mathop{\widetilde{\zeta}}{}}

\newcommand{\ooe}{\vec{{\mathrm{o}}}\mathrm{e}}

\newsavebox{\overlongequation}
\newenvironment{overlong}
 {\begin{displaymath}\begin{lrbox}{\overlongequation}$\displaystyle}
 {$\end{lrbox}\makebox[0pt]{\usebox{\overlongequation}}\end{displaymath}}   
  
\usepackage[normalem]{ulem}
\usepackage{amssymb}

\DeclareMathOperator{\rev}{rev}
\renewcommand{\vec}[1]{\uline{{\boldsymbol #1}}}

\usepackage{mathtools}
\usepackage{thmtools}

\declaretheorem[
style=plain,
name=Theorem,
numberwithin=section,
refname={Theorem,Theorems},
Refname={Theorem,Theorems}
]{Thm}
\declaretheorem[
style=plain,
name=Proposition,
numberlike=Thm,
refname={Proposition,Propositions},
Refname={Proposition,Propositions}
]{Prop}

\declaretheorem[
style=definition,
name=Remark,
numberlike=Thm,
refname={Remark,Remarks},
Refname={Remark,Remarks},
]{Rem}

\newcommand{\C}{\mathbb{C}}

\newcommand{\wwt}{{\operatorname{wt}}}
\newcommand{\hht}{{\operatorname{ht}}}
\newcommand{\ddp}{{\operatorname{dp}}}

\DeclareMathOperator{\res}{res}

\DeclareMathOperator{\Li}{Li}

\newcommand{\ii}{\mathrm{i}}

\let\epsilon\varepsilon
\let\eps\varepsilon

\newcommand{\ev}{\mathrm{ev}}
\newcommand{\od}{\mathrm{od}}

\newcommand{\sgnarg}[2]{
	\!\left( \genfrac{}{}{0pt}{0}{#1}{#2} \right)
}

\usepackage{xfrac}
\newcommand{\half}{{\sfrac{1\!}{2}}}

\newcommand{\ff}[2]{\genfrac[]{0pt}{}{#1}{#2}}
\newcommand{\st}[2]{\genfrac..{0pt}{}{#1}{#2}}

\DeclareMathOperator{\B}{B}

\newcommand{\upf}{\mathrm{f}}
\newcommand{\fininf}{{\upf,\infty}}
\newcommand{\finbul}{{\upf,\bullet}}
\newcommand{\finfin}{{\upf,\upf}}
\newcommand{\stL}[2]{\mathrel{}\middle|\mathrel{} \st{#1}{#2} }
\newcommand{\stR}[2]{\st{#1}{#2} \mathrel{}\middle|\mathrel{}}

\let\overlineO\overline
\renewcommand{\overline}[1]{\overlineO{{\vphantom{I}}#1}}
\newcommand{\ol}{\overline}

\newmuskip\pFqmuskip

\newcommand*\pFq[6][8]{%
	\begingroup 
	\pFqmuskip=#1mu\relax
	\mathchardef\normalcomma=\mathcode`,
	\mathcode`\,=\string"8000
	\begingroup\lccode`\~=`\,
	\lowercase{\endgroup\let~}\pFqcomma
	{}_{#2}F_{#3}{\left[\genfrac..{0pt}{}{#4}{#5};#6\right]}%
	\endgroup
}
\newcommand{\pFqcomma}{{\normalcomma}\mskip\pFqmuskip}

\newcommand{\vdotsB}{\raisebox{-2pt}{$\vdots$}}
\newcommand{\ddotsB}{\raisebox{-2pt}{$\ddots$}}

\usepackage{enumitem}
\setlist[enumerate,1]{leftmargin=*, label={\normalfont(\roman*)}}

\begin{document}
	
	\title[Galois descent for M$t$V's of maximal height]{An explicit Galois descent for \\ multiple $t$-values of maximal height}
	
	\author{Steven Charlton}
    \address{Department of Mathematics and Computer Science, Division of Mathematics, University of Cologne,
Weyertal 86-90, 50931 Cologne, Germany}
    \email{steven.charlton@uni-koeln.de, charlton@mpim-bonn.mpg.de}
	
	\author{Michael E. Hoffman}
	\address{Department of Mathematics, U.S. Naval Academy, Annapolis, MD 21402, USA}
	\email{meh@usna.edu}
    
	\author{Nobuo Sato}
	\address{Department of Mathematics, National Taiwan University, No. 1, Sec. 4, Roosevelt Rd., Taipei 10617, Taiwan (R.O.C.)}
	\email{nbsato@ntu.edu.tw}
	

	\date{May 11, 2026}
	
	\begin{abstract}
		We give an explicit formula for the Galois descent expressing multiple  $t$-values of maximal height in terms of classical multiple zeta values, making precise Murakami's earlier motivic result.  Our results rely on the theory of iterated beta integrals.  We apply this formula to obtain evaluations of various multiple zeta-half values.  
	\end{abstract}
	
	\maketitle

    \section{Introduction}

	The multiple series
	\[
		\Li_{k_1,\ldots,k_d}(x_1,\ldots,x_d) \coloneqq \sum_{\substack{0<m_{1}<m_{2}<\cdots<m_{d}}
		}\frac{x_{1}^{m_{1}}x_{2}^{m_{2}}\cdots x_{d}^{m_{d}}}{m_{1}^{k_{1}}m_{2}^{k_{2}}\cdots m_{d}^{k_{d}}} \,,
	\]
	are convergent in the unit polydisc \( \lvert x_1 \rvert, \ldots, \lvert x_d \rvert < 1 \), and define a family of multivariable holomorphic functions called multiple polylogarithms.   For an index \( \Bbbk = (k_1,\ldots,k_d) \), we define
	\begin{itemize}
		\item the weight as \( \wwt(\Bbbk) = k_1 + \cdots + k_d \), the sum of all arguments,
		\item the depth as \( \ddp(\Bbbk) = d \), the number of arguments,
		\item the height as \( \hht(\Bbbk) = \#\{ i \mid k_i \geq 2 \} \), the number of arguments \( \geq 2 \).
	\end{itemize}
	When all \( x_i = 1 \), these series define the \emph{multiple zeta values} (MZV's)
	\[
		\zeta(k_1,\ldots,k_d) = \sum_{\substack{0<m_{1}<m_{2}<\cdots<m_{d}}
		}\frac{1}{m_{1}^{k_{1}}m_{2}^{k_{2}}\cdots m_{d}^{k_{d}}} \qquad (\eqqcolon \Li_{k_1,\ldots,k_d}(1,\ldots,1))\,,
	\]
    which are well-studied and whose algebraic and number-theoretic properties
    are of great interest, both to mathematicians \cite{Cartier02} and to physicists
    \cite{knotsKreimer}.  When all \( x_i =\eps\in\{-1,1\}\), one obtains
    the \emph{alternating multiple zeta values} (AMZV's)
	\[
		\zeta\sgnarg{k_1,\ldots,k_d}{\eps_1,\ldots,\eps_d} = \sum_{0 < m_1 < m_2 < \cdots < m_d} \frac{\eps_1^{m_1} \cdots \eps_d^{m_d}}{m_1^{k_1} \cdots m_d^{k_d}} \qquad (\eqqcolon \Li_{k_1,\ldots,k_d}(\eps_1,\ldots,\eps_d))\,,
	\]
	which are the first example of higher-level multiple zeta values; here the level is 2, as one invokes second roots of unity.  
    We note that AMZV's appeared in the particle-physics literature
    already in the 1980's \cite{Broadhurst86}.
    Typically one simplifies the notation of AMZV's, and writes \( \overline{k_i} \) exactly when \( \eps_i = -1 \).  For example
	\[
		\zeta(k_1, \overline{k_2}, \overline{k_3}) \coloneqq  \zeta\sgnarg{k_1,k_2,k_3}{\,1, -1, -1} \,.
	\]
	
	A different type of level 2 sum is given by the multiple $t$-values (M$t$V's) \cite{HoffmanOdd19} whose denominators are restricted to 
    the odd integers:
	\[
		t(k_1,\ldots,k_d) 
		= \sum_{\substack{0<m_{1}<m_{2}<\cdots<m_{d}}} \frac{1}{(2m_{1}-1)^{k_{1}}(2m_{2}-1)^{k_{2}}\cdots (2m_{d}-1)^{k_{d}}}
		= \sum_{\substack{0<\ell_{1}<\ell_{2}<\cdots<\ell_{d} \\ \text{$\ell_i$ odd}}} \frac{1}{\ell_{1}^{k_{1}}\ell_{2}^{k_{2}}\cdots \ell_{d}^{k_{d}}} \,.
		\]
	By inserting factors \( \frac{1}{2} (1 - (-1)^{\ell_i}) \) into the numerator, and extending the sum to all \( \ell_i \), one directly obtains an expression for the M$t$V's in terms of AMZV's
	\begin{equation}\label{eqn:ttoz}
		t(k_1,\ldots,k_d) = \frac{1}{2^d} \sum_{\eps_1, \ldots,\eps_d \in \{ \pm 1 \}} \eps_1 \cdots \eps_d \, \zeta\sgnarg{k_1,\ldots,k_d}{\eps_1,\ldots,\eps_d} \,.
	\end{equation}
	A surprising fact, established by Murakami \cite[Theorem 1]{MurakamiMtVs21}, is that when the height is maximal, i.e. all \( k_i \geq 2 \), the  M$t$V \( t(k_1,\ldots,k_d) \) actually reduces to level 1.  
	
	\begin{Thm}[Murakami, Theorem 1, \cite{MurakamiMtVs21}]
		When all \( k_i \geq 2 \), \( t(k_1,\ldots,k_d) \) is a \(\mathbb{Q}\)-linear combination of multiple zeta values.
	\end{Thm}
	
	Murakami proved this result by abstract means, using the motivic Galois descent criterion developed by Glanois \cite{GlanoisThesis16, GlanoisBasis16}, without yielding an explicit expression for such M$t$V's via classical MZV's.  Our main result, \autoref{thm:main:desc}, gives an explicit formula from a more elementary standpoint.\medskip
	
	We briefly recall the motivic side of the story.  Multiple polylogarithms admit a straightforward representation via the \emph{iterated integrals}, which are defined by
	\begin{equation}\label{eqn:intro:itint}
	I(z_{0};z_{1},\ldots,z_{n};z_{n+1})\coloneqq\int_{z_{0}<t_{1}<t_{2}<\cdots<t_{n}<z_{n+1}}\frac{\dd t_{1}}{t_{1}-z_{1}}\frac{\dd t_{2}}{t_{2}-z_{2}}\cdots\frac{\dd t_{n}}{t_{n}-z_{n}} \,.
	\end{equation}
	In this notation
	\[
		\Li_{k_1,\ldots,k_d}(x_1,\ldots,x_d) = (-1)^d I(0; y_1, \{0\}^{k_1-1}, y_2, \{0\}^{k_2-1}, \ldots, y_d, \{0\}^{k_d-1}; 1) \,,
	\]
	where \( y_i = 1 \big/ \prod_{j=i}^d x_j \), for \( 1 \leq j \leq d \), and \( \{a\}^n \) denotes the string \( {a, \ldots, a} \) with \( n \) repetitions.  This extends the domain of definition of multiple polylogarithms to the universal coverings of the punctured projective line and allows one to lift multiple polylogarithms to motivic versions \cite{GoncharovMultiple01,BrownMotivicPeriodNotes}.  For \( x_i = 1 \), one obtains MZV's, which are well-known \cite{BrownMTM12} to be unramified periods of mixed Tate motives over \( \mathbb{Q} \).  The AMZV's, corresponding to \( x_i = \pm 1 \), are periods of mixed Tate motives over \( \mathbb{Q} \), where ramification at the prime 2 is allowed.  Motivic multiple zeta values (including cyclotomic ones) form a Hopf algebra comodule, with a coaction \( \Delta \).  The weight-graded pieces \( D_r = (\pi_r \otimes \id) \circ \Delta' \) of the coaction are sufficient to keep track of all the information in the coaction. 
	
	Glanois \cite[Corollary 5.1.3]{GlanoisThesis16}, \cite[Corollary 2.4]{GlanoisBasis16} established a criterion to detect when a motivic alternating MZV \( \mathfrak{Z}\) descends to classical MZV's: namely, \( D_1 \mathfrak{Z} = 0 \) and \( D_{2r+1} \mathfrak{Z} \) is already level 1.  Murakami used this criterion to inductively verify that \( t(k_1,\ldots,k_d) \), \( k_i \geq 2 \), is a linear combination of multiple zeta values, using the fact that \( t(2) = \frac{3}{4} \zeta(2) \) as the base of induction.  In a given case, one can obtain an expression by matching the coaction of the M$t$V with a linear combination of MZV's, in the Hoffman basis say.  However, the coaction cannot fix the rational multiple of \( \zeta(\wwt(\Bbbk)) \) in such a combination (cf. \cite[Theorem 2.4.4]{GlanoisBasis16} and \cite[Theorem 3.3]{BrownMTM12}), leaving a non-explicit identity overall.
	
	\medskip
	To precisely state our formula, we need to introduce some notation. 
	
	\subsection*{Normalized values}  Recall the interpolated multiple zeta values \cite{YamamotoInterpolation13}), defined by
	\begin{equation}\label{eqn:zint}
		\zeta^r(k_1,\ldots,k_d) = \sum_{0 < m_1 \leq m_2 \leq \cdots \leq m_d} \frac{r^{d - \#\{m_1,\ldots,m_d\}}}{m_1^{k_1} m_2^{k_2} \cdots m_d^{k_d}} \,.
	\end{equation}
	The case \( r = 1 \) defines the multiple zeta-star values, while the case \( r = 0 \) (with the convention \( 0^0= 1 \)) reduces to the usual multiple zeta values.  Generally \( \zeta^r \) interpolates between these values; the midway point \( r = \half \), which defines the multiple zeta-half values \( \zeta^\half \), plays an important role.  To simplify our formula, it is convenient to introduce the following rescalings,
	\begin{align*}
		\ztw^\half(\Bbbk) \coloneqq  2^{\ddp(\Bbbk)} \zeta^\half(\Bbbk) \,, \quad 
		\ttw(\Bbbk) \coloneqq 2^{\wwt(\Bbbk)} t(\Bbbk) \,.
	\end{align*}

	\subsection*{Block decomposition} We also need the notation of the block decomposition of an index (cf. \cite{CharltonBlock21}, \cite[\S2.2]{CharltonThesis}).  Let 
	\[
	\mathbb{I}\coloneqq \{ (k_{1},\ldots,k_{d}) \mid d\geq0,k_{1},\ldots,k_{d}\in\mathbb{Z}_{\geq1},k_{d} \geq2 \} 
	\]
	be the set of admissible indices and let
	\[
	\mathbb{W}_{0,1}\coloneqq \{ (a_{1},\ldots,a_{k}) \mid k \geq0,a_{1},\ldots,a_{k}\in\{0,1\}, a_{1} = 1, a_{k} = 0 \} 
	\subsetneq \smash[t]{\bigcup_{k=0}^{\infty}} \{0,1\}^{k} \,,
	\]
	be the set of binary words starting with 1 and ending with 0.
	We recall two bijective mappings $\lambda,\beta:\mathbb{W}_{0,1}\rightarrow\mathbb{I}$, called the \emph{index} and \emph{block index} respectively.
	For \( \vec{s} = (a_1,\ldots,a_k) \), the \emph{index} \( \lambda(\vec{s}) \) is defined to be the sequence of lengths of consecutive subsequences of \( \vec{s} \), obtained by splitting between $a_i$ and $a_{i+1}$ whenever \( a_{i+1} = 1\).  For example, if \( \vec{s}= (1,0,1,0,0,1,0,0,0,1,0,0,0,0)\), we find
	\[
		\lambda(\vec{s}) = \lambda(\overbrace{1,0}^{2}|\overbrace{1,0,0}^{3}|\overbrace{1,0,0,0}^{4}|\overbrace{1,0,0,0,0}^{5})=(2,3,4,5) \,.
	\]
	For \( \vec{s} = (a_1,\ldots,a_k) \), the \emph{block index} \( \beta(\vec{s}) \) is defined to be the sequences of lengths of consecutive subsequences of \( \vec{s}' \coloneqq (a_2,\ldots,a_k,1) \) (recall: \( a_1 = 1 \)), obtained by splitting between $a_i$ and $a_{i+1}$ whenever \( a_{i+1} = a_i \).  As before, we find
	\[
	\beta(\vec{s}) = \beta(\underset{\mathclap{\substack{\uparrow \\ \text{removed}}}}{\cancel{\,1\,}},\overbrace{0,1,0}^{\text{$3$}}|\overbrace{0,1,0}^{3}|\overbrace{0}^{1}|\overbrace{0,1,0}^{3}|\overbrace{0}^{1}|\overbrace{0}^{1}|\overbrace{0,\underset{\mathclap{\substack{\uparrow \\ \text{inserted}}}}{\fbox{1}}}^{\text{$2$}}) = (3,3,1,3,1,1,2) \,.
	\]
	Notice that the maps \( \lambda \) and \( \beta \) are both invertible, as their inverses can easily be written down.  We then define the \emph{block decomposition} to be \( \bl = \beta \circ \lambda^{-1} \).  From the previous examples, we find
	\[
	\bl(2,3,4,5)=(3,3,1,3,1,1,2) \,.
	\]

	With the map \( \bl \) defined, our first identity is the following.
	
\begin{Thm}[Convolution formula] \label{thm:main:conv}
	For $d\geq1$ and integers $k_{1},\ldots,k_{d}\geq2$, we have the following (note the reversed arguments in \( \zeta^\star \) and \( \bl \))
	\[
		- \ztw^\half(\bl(k_d,\ldots,k_1)) = \sum_{i=0}^d (-1)^i \ttw(k_1,\ldots,k_i) \zeta^\star(k_d,\ldots,k_{i+1}) \,.
	\]
	\end{Thm}

	For example, since \(  \bl(2,3,4,5)=(3,3,1,3,1,1,2) \), we have 
	\begin{align*}
	{-}&\ztw^\half(3,3,1,3,1,1,2) = \\
	& \zeta^\star(2,3,4,5) - \ttw(5) \zeta^\star(2,3,4) + \ttw(5, 4) \zeta^\star(2,3) - \ttw(5,4,3) \zeta^\star(2) + \ttw(5,4,3,2) \,.
	\end{align*}
	This equality is easily verified numerically, with both sides equaling \( \approx -25.79239988\ldots \).  In weight \( \leq12 \) one can even check the results symbolically with the Multiple Zeta Value Data Mine \cite{mzvDM}.
		
	\begin{Rem}[Formula for alternating zeta-half values]
		\label{rem:twoone}
		We note that there is a similar formula relating certain alternating zeta-half values directly to zeta-star values.  Let
		\[
			B(s_1,\ldots,s_k) = ((-1)^{s_1+1} \diamond s_1, \ldots, (-1)^{s_k+1} \diamond s_k) \,,
		\]
		where \( 1 \diamond x = x \) and \( -1 \diamond x = \overline{x} \), for \( x \in \mathbb{Z}_{\geq1} \).  That is to say, \( B(s_1,\ldots,s_k) \) inserts a bar over the even arguments, and leaves the odd arguments unchanged.  Then Zhao's generalized two-one theorem \cite{ZhaoIdentity16} (via the formulation given in \cite{CharltonCyclic20}) states
		\[
			\ztw^\half(B(\bl(k_1,\ldots,k_d)) = \eps_{k_1} \zeta^\star(k_1,\ldots,k_d) \,, \quad \text{ with } \quad \eps_{k_1} = \begin{cases} \phantom{-}1, & \text{if $k_1 = 1$,} \\ -1, & \text{if $k_1 > 1$} \,. \end{cases}
		\]
		For example, since \( \bl(2,3,4,5)=(3,3,1,3,1,1,2) \), we obtain \(
			\ztw^\half(3,3,1,3,1,1,\overline{2}) = \zeta^\star(2,3,4,5) 
		\).
	\end{Rem}
		
	Returning to the convolution formula \autoref{thm:main:conv}, we can obtain an expression for \( \ttw(k_1,\ldots,k_d) \), via multiple zeta values.  	\Autoref{thm:main:conv} is equivalent to the following matrix equality
	{\small
		\arraycolsep=0.5ex
		\def\arraystretch{1}
		\begin{align*}
		& \!\! \left[ \! \begin{array}{c}
		 -\ztw^\half(\bl(k_{d},\ldots,k_{1})) \\
		\vphantom{\vdots} -\ztw^\half(\bl(k_{d-1},\ldots,k_{1})) \\
		\vphantom{\vdots} \vdotsB \\
		\vphantom{\vdots} -\ztw^\half(\bl(k_{1}))\\
		\vphantom{\vdots} 1
		\end{array} \! \right] 
		=\left[\begin{array}{ccccc}
		  \vphantom{\ztw^\half} 1 \,\, & \zeta^\star(k_{d}) \,\, & \zeta^\star(k_{d},k_{d-1}) & \!\cdots\! & \zeta^\star(k_{d},\ldots,k_{1})\\
		\vphantom{\vdots} & 1 & \zeta^\star(k_{d-1}) & \!\cdots\! & \zeta^\star(k_{d-1},\ldots,k_{1})\\
		\vphantom{\vdots} &  & 1 & & \vdotsB \\
		\vphantom{\vdots} &  &  & \!\ddotsB\! & \zeta^\star(k_{1})\\
		\vphantom{\vdots} 0 &  &  &  & 1
		\end{array}\right]
		\!\! \left[ \! \begin{array}{c}
		  \vphantom{\ztw^\half} (-1)^{d}\ttw(k_{1},\ldots,k_{d})\\
		\vphantom{\vdots} (-1)^{d-1}\ttw(k_{1},\ldots,k_{d-1})\\
		\vphantom{\vdots} \vdotsB \\
		\vphantom{\vdots} -\ttw(k_{1})\\
		\vphantom{\vdots} 1
		\end{array} \! \right] \!.
		\end{align*}}%
    For any \( \Bbbk = (k_1,\ldots,k_d) \), the ``stuffle-antipode'' identity \( \sum_{i=0}^d (-1)^i \zeta(k_1,\ldots,k_i) \zeta^\star(k_d,\ldots, k_{i+1}) = 0 \) holds (see \cite[Lemma 3.3]{GlanoisBasis16} or \cite[Theorem 1.3]{CharltonHoffmanSym}).  Hence, the inverse of the zeta-star matrix is
    {\small
		\arraycolsep=0.5ex
		\def\arraystretch{1}
    \begin{overlong}
     \!\! \left[\begin{array}{ccccc}
		  \vphantom{\ztw^\half} 1 \,\, & \zeta^\star(k_{d}) \,\, & \zeta^\star(k_{d},k_{d-1}) & \!\cdots\! & \zeta^\star(k_{d},\ldots,k_{1})\\
		\vphantom{\vdots} & 1 & \zeta^\star(k_{d-1}) & \!\cdots\! & \zeta^\star(k_{d-1},\ldots,k_{1})\\
		\vphantom{\vdots} &  & 1 & & \vdotsB \\
		\vphantom{\vdots} &  &  & \!\ddotsB\! & \zeta^\star(k_{1})\\
		\vphantom{\vdots} 0 &  &  &  & 1
		\end{array}\right]^{\text{\small $-1$}} \!\!\!
         =
        \left[\begin{array}{ccccc}
		  \vphantom{\ztw^\half} 1 \,\, & (-1)^1  \zeta(k_{d}) \,\, & (-1)^2  \zeta(k_{d-1},k_{d}) & \!\cdots\! & (-1)^d \zeta(k_{1},\ldots,k_{d})\\
		\vphantom{\vdots} & 1 & (-1)^1  \zeta(k_{d-1}) & \!\cdots\! & (-1)^{d\!{-}\!1}  \zeta(k_{1},\ldots,k_{d-1})\\
		\vphantom{\vdots} &  & 1 & & \vdotsB \\
		\vphantom{\vdots} &  &  & \!\ddotsB\! & (-1)^1  \zeta(k_{1})\\
		\vphantom{\vdots} 0 &  &  &  & 1
		\end{array}\right] \!.
    \end{overlong}}
        
    In particular, we have following matrix equality
    {\small
		\arraycolsep=0.5ex
		\def\arraystretch{1}
       \begin{overlong}
         \!\! \left[ \! \begin{array}{c}
		\vphantom{\ztw^\half} (-1)^{d}\ttw(k_{1},\ldots,k_{d})  \\
		\vphantom{\vdots} (-1)^{d-1}\ttw(k_{1},\ldots,k_{d-1}) \\
		\vphantom{\vdots} \vdotsB  \\
		\vphantom{\vdots} -\ttw(k_{1}) \\
		\vphantom{\vdots} 1 
		\end{array} \! \right] 
          = 
		 \left[\begin{array}{ccccc}
		  \vphantom{\ztw^\half} 1 \,\, & (-1)^1 \, \zeta(k_{d}) \,\, & (-1)^2 \, \zeta(k_{d-1},k_{d}) & \!\cdots\! & (-1)^d \,\zeta(k_{1},\ldots,k_{d})\\
		\vphantom{\vdots} & 1 & (-1)^1 \, \zeta(k_{d-1}) & \!\cdots\! & (-1)^{d\!{-}\!1} \, \zeta(k_{1},\ldots,k_{d-1})\\
		\vphantom{\vdots} &  & 1 & & \vdotsB \\
		\vphantom{\vdots} &  &  & \!\ddotsB\! & (-1)^1 \, \zeta(k_{1})\\
		\vphantom{\vdots} 0 &  &  &  & 1
		\end{array}\right] \!\! \left[ \! \begin{array}{c}
		 -\ztw^\half(\bl(k_{d},\ldots,k_{1})) \\
		\vphantom{\vdots} -\ztw^\half(\bl(k_{d-1},\ldots,k_{1})) \\
		\vphantom{\vdots} \vdotsB \\
		\vphantom{\vdots} -\ztw^\half(\bl(k_{1}))\\
		\vphantom{\vdots} 1
		\end{array} \! \right] \!.
		\end{overlong}}%
	Extracting the identity from the first row, and multiplying by \( (-1)^d \), then gives the following theorem.
		
\begin{Thm}[Galois descent]\label{thm:main:desc}
	For $d\geq1$ and integers $k_{1},\ldots,k_{d}\geq2$, we have
    \[
        \ttw(k_{1},\ldots,k_{d}) = \zeta(k_1,\ldots,k_d) + \sum_{i=1}^d (-1)^{i-1} \zeta(k_{i+1}, \ldots, k_d) \ztw^\half(\bl(k_i,\ldots,k_1))
    \]
	In particular, the right-hand side is expressed purely via MZV's, hence
	the formula gives an explicit Galois descent for multiple $t$-values
	of maximal height. 
\end{Thm}

	\begin{Rem}[Evaluation of \( t(2, \ldots, 2, 3, 2, \ldots, 2) \)]
		Since \( \bl(\{2\}^a, 3, \{2\}^b) = (2a + 1, 2b+2) \), we note that Murakami's hypergeometric evaluation of \( t(2, \ldots, 2, 3, 2, \ldots, 2) \) \cite[Theorem 3]{MurakamiMtVs21} can be derived readily from \autoref{thm:main:conv}.   This evaluation is the key input to the proof that \( t(k_1,\ldots,k_d) \), \( k_i \in \{ 2, 3 \} \), is a basis for motivic multiple \emph{zeta} values \cite[Theorem 2]{MurakamiMtVs21}. 
		
		To show this evaluation, use \autoref{thm:main:conv} to write a generating series identity for \( t(2, \ldots, 2, 3, 2, \ldots, 2) \) in terms of the generating series for \( \ztw^\half(2a+1, 2b+2) \) and for \( \zeta^\star(2, \ldots, 2, 3, 2, \ldots, 2) \).  Use the parity theorem, given explicitly in \cite[\S3]{PanzerParity16}, to reduce \( \ztw^\half(2a+1, 2b+2) \) to a polynomial in single zeta values as a generating series identity; together with Zagier's generating series evaluation of \( \zeta^\star(2, \ldots, 2, 3, 2, \ldots, 2) \) \cite[Theorem 4, proof]{zagier2232}, one obtains the required expression for the M$t$V generating series.  Other proofs of Murakami's evaluation are already known, going first by an evaluation for \( t^\star(2, \ldots, 2, 3, 2, \ldots, 2) \) in \cite[Theorem 4.4, and (4.5)]{ligen24} (using another variant of a two-one formula for M$t$V's in terms of Mt$\half$V's) and then via the usual ``stuffle-antipode" considerations \cite[Proposition 4.3]{liLevel24} to obtain an evaluation for \( t(2, \ldots, 2, 3, 2, \ldots, 2) \).
	\end{Rem}

	Finally, let us take a moment to understand the structure of \( \bl(k_1,\ldots,k_d) \), when \( \Bbbk = (k_1,\ldots,k_d) \) has maximal height.  Since \( k_1,\ldots,k_d \geq 2 \), we can uniquely write
	\[
		\Bbbk = (\{2\}^{a_0}, 2+b_1, \{2\}^{a_1}, 2 + b_2, \ldots, \{2\}^{a_{r-1}}, 2 + b_r, \{2\}^{a_r} \} \,, \quad a_0,\ldots,a_r \geq 0, b_1,\ldots,b_r \geq 1 \,.
	\]
	So the corresponding 01-sequence \( \lambda^{-1}(\Bbbk) \in \mathbb{W}_{0,1} \) is given by
	\[
		\vec{s} = \lambda^{-1}(\Bbbk) = (\{1,0\}^{a_0} \mid 1, \{0\}^{b_1+1} \mid \{1,0\}^{a_1} \mid 1, \{0\}^{b_2+1} \mid \cdots \mid \{1,0\}^{a_{r-1}} \mid 1, \{0\}^{b_r+1} \mid \{ 1, 0 \}^{a_r}) \,.
	\]
	One then has
	\begin{align*}
		\vec{s}' &= (0, \{1,0\}^{a_0} \mid \{0\}^{b_1} \mid \{1,0\}^{a_1} \mid 1, \{0\}^{b_2+1} \mid \cdots \mid \{1,0\}^{a_{r-1}} \mid 1, \{0\}^{b_r+1} \mid \{ 1, 0 \}^{a_r}, 1) \\
		&= (0, \{1,0\}^{a_0} \mid \{0\}^{b_1-1} \mid 0, \{1,0\}^{a_1+1} \mid \{0\}^{b_2-1} \mid 0, 1, \ldots, 0 \mid 0, \{1,0\}^{a_{r-1}+1} \mid \{0\}^{b_r-1} \mid \{ 0, 1\}^{a_r+1}) \,,
	\end{align*}
	whence one can read off the block index \( \beta(\vec{s}) \) as
	\[
		\bl(\Bbbk) = \beta(\vec{s}) = \begin{cases}
			(2a_0+1, \{1\}^{b_1-1}, 2a_1 + 3, \{1\}^{b_2-1}, \ldots, 2a_{r-1} + 3, \{1\}^{b_r-1}, 2a_r+2) & \text{if $r \geq 1$,} \\
			(2a_0) & \text{if $r=0$.}
		\end{cases}
	\]
	We see that the ``maximal-height'' condition \( k_1,\ldots,k_d \geq 2 \) corresponds to the condition that \( \bl(\Bbbk) \) has only odd entries, except for the last entry which is always even.  In other words, \( \bl(\mathbb{I}_{\mathrm{m.h.}}) = \mathbb{I}_{\ooe} \), where
	\begin{align*}
		\mathbb{I}_{\mathrm{m.h.}} & \coloneqq \{ (k_1,\ldots,k_d) \in \mathbb{I} \mid d \geq 1, k_1,\ldots,k_d \geq 2 \} \\
		\mathbb{I}_{\ooe} & \coloneqq \{ (\ell_1,\ldots,\ell_r) \in \mathbb{I} \mid r \geq 1, \text{$\ell_1,\ldots,\ell_{r-1}$ odd, and $\ell_r$ even} \} \,.
	\end{align*}
	In particular, \autoref{thm:main:conv} and the formula in \autoref{rem:twoone} are disjoint.  Moreover, if \( k_1,\ldots,k_d \geq 3 \), then
	\[
		\bl(k_1,\ldots,k_d) = (1, \{1\}^{k_1-3}, 3, \{1\}^{k_2-3}, 3, \ldots, 3, \{1\}^{k_d-3}, 2) \,,
	\]
	so \( \bl(\mathbb{I}_{\geq3}) = \mathbb{I}_{1\{13\}2} \), where
	\begin{align*}
	\mathbb{I}_{\geq3} & \coloneqq \{ (k_1,\ldots,k_d) \in \mathbb{I} \mid d \geq 1, k_1,\ldots,k_d \geq 3 \} \subsetneq \mathbb{I}_{\mathrm{m.h.}} \\
	\mathbb{I}_{1\{13\}2} & \coloneqq \{ (\ell_1,\ldots,\ell_r) \in \mathbb{I} \mid r \geq 1, \ell_1=1, \ell_2,\ldots,\ell_{r-1} \in \{ 1,3\}, \ell_r = 2 \} \subsetneq \mathbb{I}_{\ooe} \,.
	\end{align*}

	\begin{Rem}[Problems for non-maximal height]
		As given, \autoref{thm:main:conv} and \autoref{thm:main:desc} cannot extend to the case when even a single \( k_i = 1 \) is permitted.  Already in the simplest case,
		\[
			\ttw(1,2) = 2^3 \cdot t(1,2) = -\frac{7}{2} \zeta(3) + \pi^2 \log(2) \,,
		\]
		so not admitting a Galois descent to level 1.  Allowing \( k_1 = 1 \) would also raise questions about regularization, as the terms \( \ztw^\half(\bl(k_i,\ldots,k_1)) \) could be divergent.  Nevertheless, it could be interesting to track through the appropriate modifications to the proof in \autoref{sec:mainproof}, in order to obtain a more general convolution identity.
	\end{Rem}
	
	The remainder of the paper is structured as follows.  In \autoref{sec:app} we give several direct and independent applications of the convolution formula \autoref{thm:main:conv} to the evaluation of multiple zeta-half values.  In particular, we evaluate \( \zeta^\half(1, \{ \{1\}^{a-3}, 3\}^{b-1}, \{1\}^{a-3}, 2) \) in \autoref{sec:app:z11312}, we evaluate \( \zeta^\half(\{1\}^n, 4) \) and \( \zeta^\half(3, \{1\}^{n-1}, 2) \) in \autoref{sec:app:z114}, we evaluate \( \zeta^\half(\{1\}^i, 3, \{1\}^j, 2) \), \( i+j \) even, in \autoref{sec:app:z113112}.  Finally we obtain an apparently new hypergeometric identity (which seems amenable to generalization in several possible ways) in \autoref{sec:app:hyp}.
	
	Then we turn to the details of the proof of \autoref{thm:main:conv}.  In \autoref{sec:beta} we recall the main properties of iterated integrals and how they can express multiple zeta values and multiple $t$-values.  We recall also the construction and properties of the iterated beta integrals introduced by Hirose and the third named author \cite{HiroseSatoBeta}, which will be used extensively.  In \autoref{sec:mainproof} we give the proof of \autoref{thm:main:conv} using the theory of iterated (beta) integrals, and hence establish the Galois descent formula for the multiple $t$ values \( t(k_1,\ldots,k_d) \), with \( k_1,\ldots,k_d \geq 2 \), of maximal height.

	\subsection*{Acknowledgements} We are grateful to Henrik Bachmann for identifying a further simplification of our original formulation of \autoref{thm:main:desc}.  
    
    SC and MEH are grateful to the Max Planck Institute for Mathematics, Bonn, for support, hospitality and excellent working conditions.  This work grew out of on-going discussions during MEH's extended stay at the MPIM amid the unusual circumstances of 2020.
	
	SC is grateful to the Japan Tourism Agency's ``Project to Promote the Attraction and Hosting of International Conferences at Universities'' for support to attend the ``17th MSJ-SI Developments of Multiple Zeta Values'' (10--22 February 2025). NS was supported by grant number 113-2115-M-002-007-MY3 from the National Science and Technology Council (R.O.C.); this grant funded a research visit for SC to NS at National Taiwan University (2--16 March 2025), where significant progress on this paper occurred. 
	
    MEH received partial support from the Naval Academy Research Council in 2021.  Views expressed in this paper are those of the authors
    and do not reflect the official policy of the U. S. Naval Academy, Department of the Navy, Department of War, or U. S. Government.

    \section{Applications to evaluations and identities}\label{sec:app}
    
    In this section we give several direct applications of the above theorems to the evaluation of multiple zeta-half values, by passing to more structured evaluations on the multiple $t$-value and multiple zeta-star value side.
    
    \subsection{Evaluation of \texorpdfstring{\( \zeta^\half(1, \{ \{1\}^{a-3}, 3\}^{b-1}, \{1\}^{a-3}, 2) \)}%
    	{zeta\textasciicircum{}½(1, \{ \{1\}\textasciicircum{}(a-3), 3\}\textasciicircum{}(b-1), \{1\}\textasciicircum{}(a-3), 2)}}
    \label{sec:app:z11312}

    Specialize \autoref{thm:main:conv} to the case where \( \Bbbk = (\{a\}^{b}) \), with \( a \geq 3, b \geq 1 \).  Write \( \rev(k_1,\ldots,k_d) \coloneqq (k_d,\ldots,k_1) \) for the reversed index.  Then \( \bl(\rev(\Bbbk)) = (1, \{ \{1\}^{a-3}, 3\}^{b-1}, \{1\}^{a-3}, 2) \), and we have
    \[
    - \ztw^\half(1, \{ \{1\}^{a-3}, 3\}^{b-1}, \{1\}^{a-3}, 2) = \sum_{i=0}^{b} (-1)^i \ttw(\{a\}^i) \zeta^\star(\{a\}^{b-i}) \,.
    \]
    For rescaling \( \ztw^\half \) and \( \ttw \) back to \( \zeta^\half \) and \(t \), note that \( \ddp(1, \{ \{1\}^{a-3}, 3\}^{b-1}, \{1\}^{a-3}, 2) = (a-2)b + 1 \), and \( \wwt(\{a\}^i) = ai \).  After doing this, and forming the generating series \(  \sum_{b=1}^\infty \bullet \,  x^{ab} \), we find
    \[
    	\sum_{b=1}^\infty 2^{1-2b} \zeta^\half(1, \{ \{1\}^{a-3}, 3\}^{b-1}, \{1\}^{a-3}, 2) x^{ab} = 1 - \sum_{i=0}^{\infty} (-1)^i t(\{a\}^i) x^{ai} \cdot  \sum_{i=0}^{\infty} \zeta^\star(\{a\}^{b-i}) \Big(\frac{x}{2}\Big)^{a(b-i)}    	
    \]
    We have following generating series (cf. \cite[Equation 11]{BBBkfold}, note we take the exponent of \( \lambda \) to be the MZV weight)
    \begin{equation}\label{eqn:zrep}
    \mathfrak{Z}_s(\lambda) \coloneqq \sum_{n=0}^\infty \zeta(\{s\}^n) \lambda^{sn} = \exp\Big\{ \sum_{k=1}^\infty \frac{(-1)^{k-1} \lambda^{sk}}{k} \zeta(sk) \Big\} \,.
    \end{equation}
    From the results in the last section of \cite{HoffmanIhara17}, we have the following generating series for MZSV's and M$t$V's with a repeating argument in terms of \( \mathfrak{Z}_s(\lambda) \), \begin{equation}\label{eqn:zrepgs}
    \begin{gathered}
    \sum_{n=0}^\infty \zeta^\star(\{s\}^n) x^{sn} = \frac{1}{\mathfrak{Z}_s(-x)} \,,  \quad 
    \sum_{n=0}^\infty t(\{s\}^n) x^{sn} = \frac{\mathfrak{Z}_s(x)}{\mathfrak{Z}_s(x/2)} \,.
    \end{gathered}
    \end{equation}
    Thus we have the following generating series identity:
    \[
     \sum_{b=1}^\infty 2^{1-2b} \zeta^\half(1, \{ \{1\}^{a-3}, 3\}^{b-1}, \{1\}^{a-3}, 2) x^{ab} =  1 - \frac{\mathfrak{Z}_a(-x)}{\mathfrak{Z}_a(-x/2)^2} \,.
   	\]
   	Note also that we have the following evaluation  (cf. \cite[Equation 34]{BBBkfold}) of \( \mathfrak{Z}_a(\lambda) \) when \( a = 2m \) is even:
   	\[
   	\mathfrak{Z}_{2m}(\lambda) = ( \ii \pi \lambda)^{-m} \prod_{j=1}^m \sin\Big( e^{(2j-1)  \ii \pi / (2m) }\pi \lambda \Big) \,.
   	\]
	We obtain the following proposition.
	\begin{Prop}   	
\begin{enumerate}
   	\item The following evaluation holds for \( a \geq 3, b \geq 1 \),
   	\[
   		 \zeta^\half(1, \{\{1\}^{a-3},3\}^{b-1}, \{1\}^{a-3}, 2) = 2^{2b-1} \cdot  [z^{ab}] \bigg( {-} \frac{\mathfrak{Z}_a(-z)}{\mathfrak{Z}_a(-z/2)^2} \bigg) \,.
   	\]
   	
   	\item  Moreover, the following evaluation holds if \( a = 2m \geq 4 \) is even, and \( b \geq 1 \),
	   \[
	   \zeta^\half(1, \{\{1\}^{2m-3},3\}^{b-1}, \{1\}^{2m-3}, 2) = 2^{2b-1} \cdot  [z^{2mb}] \bigg( \frac{\ii^{m+1} \pi^m z^m }{2^m} \prod_{j=0}^{m-1} \cot\Big( e^{j\mkern1mu \ii\pi / m} \frac{\pi z}{2} \Big) \! \bigg) \,.
	   \]
   \end{enumerate}
	\end{Prop}

	For example, if \( m = 3 \), then since
	\begin{align*}
	& \frac{t^3 \pi^3}{8} \cot\Big( \frac{\pi t}{2} \Big)\cot\Big( e^{\ii\pi/3} \frac{\pi t}{2} \Big)\cot\Big( e^{2\ii\pi/3}  \frac{\pi t}{2} \Big) \\
	& = -1 + \frac{31 \pi^6}{30240} t^6 + \frac{40247\pi^{12}}{2615348736000} t^{12} + \frac{1595681 \pi ^{18}}{6386367771463680000} t^{18} + \cdots \,,
	\end{align*}
	we find for \( 1 \leq b \leq 3 \), 
	\begin{align*}
	\zeta^\half(\{1\}^4, 2) & = 2^1 \cdot \frac{31 \pi^6}{30240} = \frac{31}{16} \zeta(6) \,, \\
	\zeta^\half(\{1\}^4, 3, \{1\}^3, 2) & = 2^3 \cdot \frac{40247\pi^{12}}{2615348736000} = \frac{40247}{353792} \zeta(12) \,, \\
	\zeta^\half(\{1\}^4, 3, \{1\}^3, 3, \{1\}^3, 2) & = 2^5 \cdot \frac{1595681 \pi ^{18}}{6386367771463680000} = \frac{1595681}{224599040} \zeta(18) \,.
	\end{align*}
	Each of these can be readily verified with the MZV Datamine \cite{mzvDM}.
   	
   	\subsubsection*{Expression via complete symmetric functions} We now give a somewhat more explicit evaluation as a polynomial in single zeta values.  We recall \cite[\S2, p.~21 onwards]{MacdonaldBook} the polynomial \( Q_n(x_1,x_2,\ldots,x_n) \), which expresses the $n$-th complete symmetric function \( h_n \) in terms of power sums, so for example \( Q_3(x_1,x_2,x_3) = \frac{1}{6} (x_1^3 + 3x_1 x_2 + 2 x_3) \).  By \cite[(2.14')]{MacdonaldBook} and \cite[p.~28]{MacdonaldBook} we have
   	\[
	   	 Q_n(x_1,x_2,\ldots,x_n) = \sum_{1j_1 + \cdots + nj_n = n} \prod_{i=1}^n \frac{x_i^{j_i}}{i^{j_i} j_i!} 
	   	= \frac{1}{n!} \det\left[ \begin{array}{ccccc}
	   	  x_1 & -1 & 0 & \cdots & 0 \\
	   	 \vphantom{\vdots} x_2 & x_1 & -2 & \cdots &  0 \\
	   	 \vphantom{\vdots} \vdotsB & \vdotsB & \vdotsB & \ddotsB &  \vdotsB  \\
	   	 \vphantom{\vdots} x_{n-1} & x_{n-2} & x_{n-3} & \cdots & -(n-1) \\
	   	 \vphantom{\vdots} x_n & x_{n-1} & x_{n-2} & \cdots & x_1
	   	 \end{array} \right] \,.
   	\]
	 Likewise (cf. \cite[Proof of (2.14$'$)]{MacdonaldBook}), we have the generating series 
   	\[
   	\sum_{n=0}^\infty Q_n(x_1,\ldots,x_n) \frac{t^n}{n!} = \exp\bigg( \sum_{j=1}^\infty x_j \frac{t^j}{j} \bigg)  \,.
   	\]
   	Setting \( x_j = \zeta(\overline{a j}) = - (1 - 2^{1- aj}) \zeta(a j) \) shows directly that
   	\[
   	\sum_{n=0}^\infty Q_n(\zeta(\overline{a}),\,\zeta(\overline{2a}),\,\zeta(\overline{3a}),\,\ldots,\,\zeta(\overline{an})) \frac{t^{an}}{n!} = \frac{\mathfrak{Z}_a(t)}{\mathfrak{Z}_a(t/2)^2} \,
   	\]
   	whence we obtain the following somewhat more explicit formula.	
   	\begin{Prop}
   		Let \( Q_n(x_1,x_2,\ldots,x_n) \) be the polynomial which expresses the $n$-th complete symmetric function \( h_n \) in terms of power sums \cite[\S2]{MacdonaldBook}.  The following evaluation holds, for \( a \geq 3 \), \( b \geq 1 \),
   		\[
   			\zeta^\half(1, \{\{1\}^{a-3},3\}^{b-1}, \{1\}^{a-3}, 2) = -2^{2b-1} Q_b(\zeta(\ol{a}), \zeta(\overline{2a}), \ldots, \zeta(\overline{ab}) ) \,,
   			\]
		where \( \zeta(\overline{k}) = -(1 - 2^{1-k})\zeta(k) \).
   	\end{Prop}
   
    \subsection{Evaluations of \texorpdfstring{\( \zeta^\half(\{1\}^n, 4) \)}{zeta\textasciicircum{}½(\{1\}\textasciicircum{}n, 4)} and \texorpdfstring{\( \zeta^\half(3, \{1\}^{n-1}, 2) \)}{zeta\textasciicircum{}½(3, \{1\}\textasciicircum{}(n-1), 2)}}
        \label{sec:app:z114}

	Specialize \autoref{thm:main:conv} to the case \( \Bbbk = (2,n+2) \), so that \( \bl(\rev(\Bbbk)) = (\{1\}^n, 4) \).  We obtain, after rescaling,
	\begin{equation}\label{eqn:zh1114}
		 \zeta^\half(\{1\}^n, 4) = -\frac{1}{2^{n+1}} \zeta^\star(n+2,2) + \frac{1}{2^{n-1}} t(2) \zeta^\star(n+2) - 8 t(2,n+2) \,.
	\end{equation}
	We recall from \cite[Proposition 5.3]{CharltonKeilthyDouble}, that \( t(\ev, \ev) \) has a depth-preserving Galois descent to classical double zeta values.  In particular, by setting \( \ell = 1 \) in the referenced reduction for \( t(2\ell,2k) \), we find
	\begin{equation}\label{eqn:t2ev}
		t(2,2k) = 
		 \sum_{\substack{i+j = 2k+2 \\ i,j \geq 2}} \frac{i-1}{2^{i+1}} \zeta(i,j) + \frac{1}{2^{2k+2}} \zeta(2,2k)  +  \frac{1}{2^{2k+1}} \zeta(2k) \zeta(2)  - \frac{k(2k+5)}{2^{2k+3}} \zeta(2k + 2) \,.
	\end{equation}
	On the other hand, from the parity theorem for MZV's (cf. Panzer's explicit identity \cite[Equation 3.3]{PanzerParity16}) and M$t$V's (cf. the double  $t$-value reductions in \cite[Theorem 4.1, and 4.2]{XuParametric}) we have
	\begin{align}
	\label{eqn:zpar} \zeta(a,b) & {}= \begin{aligned}[t]
	& -\frac{1}{2} \zeta(a+b) - \frac{(-1)^a}{2} \Bigg\{ \binom{a+b-1}{a-1} + \binom{a+b-1}{b-1}\Bigg\} \zeta(a+b) \\
	& {} + \delta_{\text{$b$ even}} \zeta(a)\zeta(b) +  (-1)^a \sum_{\substack{2s+k = a+b \\ s > 0, k > 1}} \zeta(2s)\zeta(k) \bigg\{ \binom{k-1}{a-1} + \binom{k-1}{b-1} \bigg\} 
	\end{aligned} \\
	\label{eqn:tpar} t(a,b) & {} = \begin{aligned}[t]
	& -\frac{1}{2} (1 - 2^{-a-b}) \zeta(a+b) + \delta_{\text{$b$ even}} (1-2^{-a})(1-2^{-b}) \zeta(a)\zeta(b) \\
	&  +  (-1)^a \sum_{\substack{2s+k = a+b \\ s > 0, k > 1}} (1-2^{-2s}) \zeta(2s) \frac{\zeta(k)}{2^{a+b-2s}} \bigg\{ \binom{k-1}{a-1} + \binom{k-1}{b-1} \bigg\}
	\end{aligned}
	\end{align}
	Using these one can check that \autoref{eqn:t2ev} also holds after replacing \( k \mapsto k/2 \).  In particular, we have the uniform expression
	\[
	 t(2,k) = 
	 \sum_{\substack{i+j = k+2 \\ i,j \geq 2}} \frac{i-1}{2^{i+1}} \zeta(i,j) + \frac{1}{2^{k+2}} \zeta(2,k)  +  \frac{1}{2^{k+1}} \zeta(k) \zeta(2)  - \frac{k(k+5)}{2^{k+4}} \zeta(k + 2) \,, \quad k \geq 2
	\]
	Substituting this into \autoref{eqn:zh1114} gives us an evaluation for \( \zeta^\half(\{1\}^n,4) \) via double zeta values.  When \( n \) is odd we can apply the parity theorem in \eqref{eqn:zpar} to obtain an expression in terms of Riemann zeta values.
	\begin{Prop} 
		\begin{enumerate}
			\item 
			The following evaluation holds for \( n\geq0 \):
	\begin{align*}
		\zeta^\half(\{1\}^n,4) = {} & \!\!\! \sum_{\substack{i+j = n+4 \\ i,j \geq 2}} \frac{1-i}{2^{i-2}} \zeta(i,j) + \frac{(n+2)(n+7)}{2^{n+3}} \zeta(n+4)
	\end{align*}
	\item The following evaluation holds for \( m \geq 0 \):
	\begin{align*}
		\zeta^\half(\{1\}^{2m+1},4) = {} & \!\!\! \sum_{\substack{2i+j = 2m+5 \\ i \geq 1, j \geq 2}} \!\!\! 8 \Big( 1 - \frac{1}{2^{2i-1}} \Big) \zeta(2i) \cdot \frac{1-j}{2^j} \zeta(j) + \Big( 4 - \frac{3}{2^{2m+1}} - \frac{m(2m+9)}{2^{2m+3}} \Big) \zeta(2m+5) \,.
	\end{align*}
\end{enumerate}
	\end{Prop}

	On the other hand, specializing \autoref{thm:main:conv} to the case \( \Bbbk = (n+2,2) \), with \( \bl(\rev(\Bbbk)) = (3,\{1\}^{n-1}, 2) \), gives
\begin{equation}\label{eqn:zh3112}
\zeta^\half(3,\{1\}^{n-1}, 2) = -\frac{1}{2^{n+1}} \zeta^\star(2,n+2) + 2 t(n+2) \zeta^\star(2) - 8 t(n+2,2) \,.
\end{equation}
	Adding \autoref{eqn:zh1114} and \eqref{eqn:zh3112}, we readily find (using the stuffle product, recalling \( t(n) = (1-2^{-n})\zeta(n) \))
	\begin{align*}
		\zeta^\half(\{1\}^{n}, 4) + \zeta^\half(3,\{1\}^{n-1}, 2) = 
		-4 \Big(1 - \frac{1}{2^{n+1}} \Big) \zeta(2) \zeta(n+2) + \Big(8 - \frac{1}{2^n} \Big) \zeta(n+4) \,
	\end{align*}
	whence we immediately obtain similar evaluations for \( \zeta^\half(3,\{1\}^{n-1}, 2) \).	
	\begin{Prop} 
		\begin{enumerate}
			\item 
			The following evaluation holds for \( n\geq1 \),
			\begin{align*}
			\zeta^\half(3,\{1\}^{n-1},2) = {} & \!\!\! \sum_{\substack{i+j = n+4 \\ i,j \geq 2}} \frac{i-1}{2^{i-2}} \zeta(i,j) 
			- \Big( 4 - \frac{1}{2^{n-1}} \Big) \zeta(2) \zeta(n+2) + \Big( 8 - \frac{1}{2^n} - \frac{(n+2)(n+7)}{2^{n+3}} \Big) \zeta(n+4) \,.
			\end{align*}
			\item The following evaluation holds for \( m \geq 0 \),
			\begin{align*}
			\zeta^\half(3,\{1\}^{2m},2) = {} \!\!\! \sum_{\substack{2i+j = 2m+5 \\ i \geq 1, j \geq 2}} \!\!\! 8 \Big( 1 - \frac{1}{2^{2i-1}} \Big) \zeta(2i) \cdot \frac{j-1}{2^j} \zeta(j) - \Big( 4 - \frac{1}{2^{2m}} \Big) \zeta(2)\zeta(2m+3) & \\[-1em]
			{} + \Big( 4 + \frac{1}{2^{2m}} + \frac{m(2m+9)}{2^{2m+3}} \Big) \zeta(2m+5) & \,.
			\end{align*}
		\end{enumerate}
	\end{Prop}

	\subsection{Evaluation of \texorpdfstring{\( \zeta^\half(\{1\}^i, 3, \{1\}^j, 2) \)}{zeta\textasciicircum{}½(\{1\}\textasciicircum{}i, 3, \{1\}\textasciicircum{}j, 2)}, \texorpdfstring{\( i + j \)}{i+j} even}
	    \label{sec:app:z113112}

	Specialize \autoref{thm:main:conv} to the case \( \Bbbk = (j+3,i+2) \), so that \( \bl(\rev(\Bbbk)) = (\{1\}^{i}, 3, \{1\}^j, 2) \).  We obtain, after rescaling,
	\[
	\zeta^\half(\{1\}^i, 3, \{1\}^j, 2) = -\frac{1}{2^{i+j+2}}  \zeta^\star(i+2,j+3) + \frac{1}{2^{i-1}} \zeta(i+2) t(j+3) - 8 t(j+3, i+2) \,.
	\]
	Under the assumption that \( i + j \) is even, the parity theorem for MZV's and M$t$V's (see \autoref{eqn:zpar} and \eqref{eqn:tpar} above) reduces this to a polynomial in Riemann zeta values.  We find the following.
	\begin{Prop}
		The following evaluation holds for \( i, j \geq 0 \), with \( i + j \) even,
	\begin{align*}
		\zeta^\half(\{1\}^i, 3, \{1\}^j, 2) = 
		& \sum_{\substack{2s + k = i+j+5 \\ s \geq 1, k \geq 2}} \frac{(-1)^j (2^{2s}-1) - (-1)^i}{2^{i+j+2}} \bigg\{ \binom{k-1}{i+1} + \binom{k-1}{j+2} \bigg\} \zeta(2s) \zeta(k)  \\*[-0.5ex]
		& + \Big( \frac{2^{j+3} - 2}{2^{i+j+2}} - \delta_{\text{$i\!$ even}} (8 - 2^{1-i} - 2^{-j}) \Big)\zeta(i+2) \zeta(j+3) \\*[0.5ex]
		 & + \frac{1}{2^{i+j+3}} \bigg( 2^{i+j+5} - 2 + (-1)^i \bigg\{ \binom{i+j+4}{i+1} + \binom{i+j+4}{j+2} \bigg\} \bigg) \zeta(i+j+5) \,.
	\end{align*}
	\end{Prop}

	One can also find a formula for \( \zeta^\half(\{1\}^\ev, 3, \{1\}^\od, 2) \) in terms of double zeta values, by using the depth preserving Galois descent of \( t(\ev, \ev) \) from \cite[Proposition 5.3]{CharltonKeilthyDouble}; we leave this to the interested reader.  However, since \( t(\od,\od) \) can only be expressed via MZV's of depth \( \geq 4 \) (cf. \cite[Remark 5.2]{CharltonKeilthyDouble}), a formula for \( \zeta^\half(\{1\}^\od, 3, \{1\}^\ev, 2) \) is beyond the scope of the current work. 

	\subsection{A hypergeometric identity}
	    \label{sec:app:hyp}
	
	Specialize \autoref{thm:main:conv} to the case \( \Bbbk = (\{3\}^n,2) \), so that \( \bl(\rev(\Bbbk)) = (\{3\}^n,2) \) also (this seems to be the only such instance).  We obtain, after rescaling,
	\[
		-2^{n+1} \zeta^\half(\{3\}^n, 2) = \sum_{i=0}^n (-1)^i  2^{3i}  t(\{3\}^i) \cdot\zeta^\star(2, \{3\}^{n-i}) - (-1)^{n+1} 2^{3n+2} t(\{3\}^n, 2) \,.
	\]
	Form the generating series \( \sum_{n=0}^\infty \bullet \, x^{3n+2} \), and we obtain
	\[
		\sum_{n=0}^\infty -2^{\sfrac{1\!}{3}} \zeta^\half(\{3\}^n, 2) (2^{\sfrac{1\!}{3}} x)^{3n+2}
			=
		\sum_{n=0}^\infty t(\{3\}^n) (-2x)^{3n}  \cdot \sum_{n=0}^\infty \zeta^\star(2,\{3\}^n) x^{3n+2}
			-
		\sum_{n=0}^\infty t(\{3\}^n, 2) (-2x)^{3n+2} \,.
	\]
	
	Write \( \rho = \exp(2 \ii \pi / 3) \).  Then one can directly express the above generating series via hypergeometric functions and gamma functions.  Write the following generating series equalities using the definition of (interpolated) MZV's and M$t$V's (see \cite[Eqn. (11), (44)]{zagier2232} for MZV's and zeta star values; cf. \cite[\S2.1]{AuCtel} for the general interpolated MZV's).  After this, the evaluation via hypergeometric functions follows purely as a term-wise equality with the relevant Pochhammer symbol expression.
	
	For the \( \zeta^\half \) generating series, we have
	\begin{align*}
		 \sum_{n=0}^\infty -2^{\sfrac{1\!}{3}} \zeta^\half(\{3\}^n, 2) (2^{\sfrac{1\!}{3}} x)^{3n+2} 
		& =  -2^{\sfrac{1\!}{3}} \sum_{\ell=1}^\infty \frac{1}{\ell^2} \frac{1}{1 - \frac{1}{2}  v} \prod_{0 < n < \ell} \bigg( 1 + \frac{u}{1-\frac{1}{2} u} \bigg) \,, \quad v = \frac{(2^{\sfrac{1\!}{3}})x^3}{\ell^3} , u = \frac{2^{\sfrac{1\!}{3}} x}{n^3}  \\[1ex]
		& =  \frac{2x^2}{x^3 - 1} \pFq{4}{3}{2, 1{+}x, 1{+}\rho x, 1{+}\rho^2 x}{2{-}x, 2{-}\rho x, 2{-}\rho^2 x}{1} \,.
	\end{align*}
	For the \( t(\{3\}^n) \) generating series, we have by \autoref{eqn:zrepgs},
	\begin{align*}
		\sum_{n=0}^\infty t(\{3\}^n) (-2x)^{3n}
	    &= \frac{\mathfrak{Z}_3(-2 x)}{\mathfrak{Z}_3(-x)} = \exp\bigg( \sum_{k=1}^\infty \frac{-(2^{3k} - 1)}{k}  \zeta(3k) x^{3k} \bigg) \\
		&= \frac{\Gamma(1-x)\Gamma(1-\rho x)\Gamma(1-\rho^2 x)}{\Gamma(1 - 2 x)\Gamma(1 - 2 \rho x)\Gamma(1 - 2 \rho^2 x)} \\
		&= \frac{\pi^{\sfrac{3\!}{2}}}{\Gamma(\frac{1}{2} -  x)\Gamma(\frac{1}{2} - \rho x)\Gamma(\frac{1}{2} - \rho^2 x)} \,.
	\end{align*}
	Here the final equality follows using \( \frac{\Gamma(1-z)}{\Gamma(1-2z)} = \frac{\sqrt{\pi} 2^{2z}}{\Gamma(\frac{1}{2} - z)} \), a variant of the duplication relation for the gamma function, and the relation \( 1 + \rho + \rho^2 = 0 \).  For the \( \zeta^\star \) generating series, we have
	\begin{align*}
		\sum_{n=0}^\infty \zeta^\star(2,\{3\}^n) x^{3n+2} 
		& = x^2 \sum_{\ell=1}^\infty \frac{1}{\ell^2} \prod_{n \geq \ell}\bigg( 1 - \frac{x^3}{n^3} \bigg)^{-1} \\
		& = \Gamma(1-x)\Gamma(1-\rho x)\Gamma(1-\rho^2 x) \cdot x^2 \pFq{3}{2}{1{-}x, 1{-}\rho x, 1{-}\rho x^2}{2,2}{1} \,.
	\end{align*}
	Finally, for the \( t(\{3\}^n, 2) \) generating series, we have
	\begin{align*}
		\sum_{n=0}^\infty t(\{3\}^n, 2) (-2x)^{3n+2} 
		&= (-2x)^2 \sum_{\ell=1}^\infty \prod_{0 < n < \ell} \bigg( 1 + \frac{(-2x)^3}{(2n-1)^3} \bigg) \frac{1}{(2\ell-1)^2}  \\
		& = 4x^2 \pFq{4}{3}{1, \frac{1}{2} {-} x, \frac{1}{2} {-} \rho x, \frac{1}{2} {-} \rho^2 x}{\frac{1}{2}, \frac{3}{2}, \frac{3}{2}}{1} \,.
	\end{align*}
	
	Putting these expressions together, we obtain the following apparently new \( {}_4F_3 \) identity.
	\begin{Prop}[${}_4F_3$ hypergeometric identity]
		The following identity holds
		\begin{multline*}
			 \frac{1}{x^3 - 1} \cdot \pFq{4}{3}{2, 1{+}x, 1{+}\rho x, 1{+}\rho^2 x}{2 {-} x, 2 {-} \rho x, 2 {-} \rho^2 x}{1} 
			 \,\, = \,\,
			  - 2 \cdot \pFq{4}{3}{1, \frac{1}{2} {-} x, \frac{1}{2} {-} \rho x, \frac{1}{2} {-} \rho^2 x}{\frac{1}{2}, \frac{3}{2}, \frac{3}{2}}{1} \\[1ex]
			 {} + \frac{\pi^{\sfrac{3\!}{2}}}{2}  \frac{\Gamma(1-x)\Gamma(1-\rho x)\Gamma(1-\rho^2 x)}{\Gamma(\frac{1}{2} -  x)\Gamma(\frac{1}{2} - \rho x)\Gamma(\frac{1}{2} - \rho^2 x)} 
			 \cdot \pFq{3}{2}{1-x, 1-\rho x, 1-\rho^2 x}{2,2}{1} \,.
		\end{multline*}
	\end{Prop}

	\section{Recap of iterated integrals and the theory of iterated beta integrals}\label{sec:beta}
	
	In this section, we recall the setup and properties of general iterated integrals.  We then recall the finite and infinite iterated beta integrals and their key properties, including translation invariance. \medskip
	
	\subsection{General iterated integrals}  The theory of general iterated integrals is established by Chen \cite{ChenIteratedIntegrals}; we recall the discussion given in \cite[\S2]{BrownDecomposition12}.  Let \( M \) be a smooth manifold, and \( \gamma \colon [0,1] \to M \) a piecewise smooth path on \( M \).  A family \( \omega_1,\ldots,\omega_n \) of differential 1-forms on \( M \), can be pulled back under \( \gamma \) to the interval \( [0,1] \) giving
	\[
		\gamma^\ast(\omega_i) \coloneqq f_i(t) \mathrm{d}t \,.
	\]
	The iterated integral of \( \omega_1, \ldots, \omega_n \) along \( \gamma \) is defined by
	\[
		\int_\gamma \omega_1 \cdots \omega_n = \int_{0 < t_1 < \cdots < t_n < 1} f_1(t_1) \dd t_1 \cdots f_n(t_n) \dd t_n \,.
	\]
	When \( n = 0 \), the empty integral is defined to be 1 by convention.  The iterated integrals satisfy various properties.

	\subsubsection*{Shuffle-product} Given \( r+s \) differential 1-forms \( \omega_1,\ldots, \omega_{r+s} \), then
		\[
			\int_\gamma \omega_1 \cdots \omega_r \int_\gamma \omega_{r+1} \cdots \omega_{r+s} = \sum_{\sigma \in \Sigma(r,s)} \int_\gamma \omega_{\sigma(1)} \cdots \omega_{\sigma(r+s)} \,,
		\]
		where \( \Sigma(r,s) \coloneqq \{ \sigma \in \mathfrak{S}_n \mid \sigma(1) < \cdots < \sigma(r) \,, \sigma(r+1) < \cdots < \sigma(r+s) \} \) is the set of \( (r,s) \)-shuffles.

	\subsubsection*{Composition of paths} Let \( \alpha, \beta \colon [0,1] \to M \) be two paths, such that \( \alpha(1) = \beta(0) \), and write \( \alpha\beta \) be the composed path obtained by traversing first \( \alpha \) then \( \beta \).  Then
		\[
			\int_{\alpha\beta} \omega_1 \ldots \omega_n = \sum_{i=0}^n \int_\alpha \omega_1 \cdots \omega_i \int_\beta \omega_{i+1} \cdots \omega_n \,.
		\]

	\subsubsection*{Reversal of paths} If \( \gamma^{-1}(t) = \gamma(1-t) \) is the reversal of the path \( \gamma \), then
		\[
		\int_{\gamma^{-1}} \omega_1 \cdots \omega_n = (-1)^n \int_\gamma \omega_n \cdots \omega_1 \,.
		\]

	\subsubsection*{Functoriality} For \( f \colon M' \to M \) smooth, and \( \gamma : [0,1] \to M' \), then
		\[
			\int_\gamma f^\ast \omega_1 \cdots f^\ast \omega_n = \int_{f(\gamma)} \omega_1 \cdots \omega_n \,.
		\]

	For a path \( \gamma \) from \( a \) to \( b \), it is convenient to introduce the notation
	\[
		I_\gamma(a; \omega_1 \cdots \omega_n; b) \coloneqq \int_\gamma \omega_1 \cdots \omega_n \,.
	\]
	One can allow arbitrary linear combinations of words in the differential forms \( \omega_i \) as the argument of \( I_\gamma(a;\,\bullet\,; b) \). \medskip
	
	When \( M = \mathbb{C} \setminus S \), for some finite set of points \( S \), one has the family of closed 1-forms
	\[
		e_{x_i}(t) \coloneqq \frac{\mathrm{d}t}{t-x_i} \in \Omega^1(M) \,, \text{ for } x_i \in S \,.
	\]
	Taking a path \( \gamma \colon [0,1] \to M \), with end points \( \gamma(0) = x_0, \gamma(1) = x_{n+1} \in M \), we have
	\[
		I_\gamma(x_0; x_1,\ldots,x_n; x_{n+1}) \coloneqq I_\gamma(x_0; e_{x_1} \cdots e_{x_n}; x_{n+1}) = \int_{\gamma} \frac{\dd t}{t-x_1} \cdots \frac{\dd t}{t - x_n} \,,
	\]
	extending the notation from \eqref{eqn:intro:itint}.
	
	One can also allow \( x_0, x_{n+1} \in S \), with a suitable logarithmic regularization involving tangential basepoints, even when the integral does not converge directly (cf. \cite[\S5.16]{DG05}, \cite[\S3.8]{mzvBook}).  When the path \( \gamma \) is clear from context (such as the straight line path \( 0 \to 1 \)), we may omit gamma from the notation.
	
	\subsection{Integral representations of nested sums} We recall now how multiple zeta values, and their relevant variants, can be expressed as iterated integrals.  The following iterated integral expression for the multiple polylogarithm is obtained through term-by-term integration \cite[Theorem 2.2]{GoncharovMultiple01}, after expanding each \( \frac{1}{t-y_i} \) as a geometric series:
	\[
		\Li_{k_1,\ldots,k_d}(x_1,\ldots,x_d) = (-1)^d I(0; e_{y_1} \, e_0^{k_1-1} \, e_{y_2} \, e_0^{k_2-1}  \,\cdots \, e_{y_d}  \,e_0^{y_d-1}; 1) \,, \quad y_i = 1 / \textstyle\prod_{j=i}^d x_i \,.
	\]
	Here \( e_0^n = e_0 \cdots e_0 \) denotes the word consisting of the form \( e_0 \) repeated \( n \) times.
	
	\subsubsection*{Multiple zeta values} In particular, the multiple zeta value \( \zeta(k_1,\ldots,k_d) \) has the following integral representation, by setting \( x_i = 1 \) in the multiple polylogarithms:
	\[
		\zeta(k_1,\ldots,k_d) = (-1)^d I(0; e_1 \, e_0^{k_1-1}  \,e_1  \,e_0^{k_2-1}  \,\cdots \, e_1  \,e_0^{k_d - 1}; 1) \,.
	\]
	
	\subsubsection*{Multiple \texorpdfstring{$t$}{t}-values} Using the expression of multiple $t$-values via alternating multiple zeta values (see \eqref{eqn:ttoz}), one obtains
	\begin{align*}
		t(k_1,\ldots,k_d) &= \sum_{\epsilon_i \in \{ \pm 1 \}} \frac{\epsilon_1\cdots \epsilon_d}{2^d} I(0; e_{\epsilon_1 \cdots \epsilon_d} \, e_0^{k_1 - 1}  \, e_{\epsilon_2 \cdots \epsilon_d}  \, e_0^{k_2-1}  \, \cdots e_{\epsilon_d}  \, e_0^{k_d-1}; 1) \\
		& = \sum_{\eta_i \in \{ \pm 1 \}} \frac{\eta_1}{2^d} I(0; e_{\eta_1} \, e_0^{k_1 - 1}  \, e_{\eta_2}  \, e_0^{k_2-1}  \, \cdots e_{\eta_d}  \, e_0^{k_d-1}; 1) \,.
	\end{align*}
	Notice that
	\begin{align*}
		e_{1}(t) + e_{-1}(t) = \frac{\dd t}{t - 1} + \frac{\dd t}{t+1} = \frac{2 t \dd t}{t^2 - 1} \,, \qquad 
		e_{1}(t) - e_{-1}(t) = \frac{\dd t}{t - 1} - \frac{\dd t}{t+1} = \frac{2 \dd t}{t^2 - 1} \,.
	\end{align*}
	So each sum over \( \eta_i = \pm 1\) can be simplified, giving
	\[
		t(k_1, \ldots, k_d) = (-1)^d I(0; \tfrac{\dd t}{t^2 - 1}  \, e_0^{k_1-1} \tfrac{t \dd t}{t^2 - 1}  \, e_0^{k_2-1} \, \cdots  \,\tfrac{t \dd t}{t^2 - 1} \, e_0^{k_d-1}; 1) \,,
	\]
	where the first differential form does not have \( t \) in the numerator.  It will be convenient to substitute \( t = \sqrt{s} \); we see readily
	\begin{align*}
		\frac{t \dd t}{t^2 - 1} = \frac{\sqrt{s} \, \dd \sqrt{s}}{s - 1} = \frac{\dd s}{2(s-1)} = \frac{1}{2} e_1(s) \,, \qquad 
		\frac{\dd t}{t} = \frac{\dd \sqrt{s}}{\sqrt{s}} = \frac{1}{2s} = \frac{1}{2} e_0(s) \,.
	\end{align*}
	On the other hand
	\[
		\frac{\dd t}{t^2 - 1} = \frac{\dd \sqrt{s}}{s - 1} = \frac{\dd s}{2\sqrt{s} (s-1)} = \frac{1}{2} e_1(s) - \smash{\underbrace{\frac{\dd s}{2\sqrt{s}(\sqrt{s}+1)}}_{\mathclap{\text{holomorphic at $s=1$}}}} \,.
	\]
	After multiplying both sides by \( 2^{\wwt(\Bbbk)} \), this gives the following integral representation for the rescaled multiple $t$-values:
	\begin{equation}\label{eqn:int:ttw}
	\begin{aligned}
		\ttw(k_1,\ldots,k_d) & = 2^{\wwt(\Bbbk)} t(k_1,\ldots,k_d) \\
		& = (-1)^d  I(0; ( e_1 {-} \tfrac{\dd s}{\sqrt{s}(\sqrt{s}+1)} )  \, e_0^{k_1-1}  \, e_1  \, e_0^{k_2-1}  \, \cdots  \, e_1  \, e_0^{k_d-1}; 1) \,.
		\end{aligned}
	\end{equation}
	
	\subsubsection*{Interpolated multiple zeta values} From the definition \eqref{eqn:zint}, we have
	\begin{align*}
		\zeta^r(k_1,\ldots,k_d) = \sum_{\circ_i = \text{``$+$'' or ``$,$''}} r^{\#\{i \mid \circ_i = {+}\}} \zeta(k_1 \circ_1 k_2 \circ_2 \cdots \circ_{d-1} k_d) .
	\end{align*}
	The iterated integral representations of \( \zeta(k_1,\ldots,k_i,k_{i+1},\ldots,k_d) \) and \( r \zeta(k_1,\ldots,k_i + k_{i+1}, \ldots, k_d) \) are related as follows: the form \( e_1 \) at the start of \( e_1 e_0^{k_{i+1}-1} \) becomes \( (-r e_0) \), with an extra sign from the reduction in depth. Combining all such cases, for \( e_1 e_0^{k_i} \), \( 2 \leq i \leq d \), gives
	\begin{equation}\label{eqn:zint:int}
		\zeta^r(k_1,\ldots,k_d) = (-1)^d I(0; e_1 \, e_0^{k_1-1} \,(e_1 - r e_0)\, e_0^{k_2-1} \,\cdots\, (e_1 - r e_0)\, e_0^{k_d - 1} ; 1) \,.
	\end{equation}
	In particular, for \( r = 1 \), we obtain an integral representation for the multiple zeta-star values,
	\begin{equation}\label{eqn:zstar:int}
		\zeta^\star(k_1,\ldots,k_d) = (-1)^d I(0; e_1 \,e_0^{k_1-1}\, (e_1 - e_0)\, e_0^{k_2-1}\, \cdots\, (e_1 - e_0) \,e_0^{k_d - 1} ; 1) \,.
	\end{equation}
	Likewise, for \( r = \half \),  after multiplying both sides by \( 2^{\ddp(\Bbbk)} \), we obtain an integral representation for the rescaled zeta-half values,
	\begin{equation}\label{eqn:zhalf:int}
	\begin{aligned}
		\ztw^\half(k_1,\ldots,k_d) &= 2^{\ddp(\Bbbk)} \zeta^\half(k_1,\ldots,k_d) \\
		&  = (-1)^d I(0; 2e_1\, e_0^{k_1-1}\, (2e_1 - e_0)\, e_0^{k_2-1}\, \cdots\, (2e_1 - e_0)  \, e_0^{k_d - 1} ; 1) \,.
		\end{aligned}
	\end{equation}
	
	\subsection{The iterated beta integrals}  We now recall the theory of iterated beta integrals introduced by Hirose and the third named author \cite{HiroseSatoBeta}.  The classical beta function is
	\[
		B(z_1,z_2) = \int_0^1 t^{z_1-1} (1-t)^{z_2-1} \dd t = \frac{\Gamma(z_1)\Gamma(z_2)}{\Gamma(z_1+z_2)} \,.
	\]
	To generalize this, introduce the differential form
	\[
		\ff{x,y}{\alpha,\beta}(t) \coloneqq \frac{\dd t}{(t-x)^\alpha (t-y)^{1-\beta}} \,.
	\]
	We define the iterated beta integral, and some variants, as follows.
	
	\subsubsection*{Finite-finite iterated beta integral}
	
	For points \( x_0,\ldots,x_n \in \C \), with \( x_i \neq x_{i+1} \), parameters \( \alpha_0,\ldots,\alpha_n \in \C \) and some fixed path \( \gamma \) from \( p \in \C\) to \( x_n \in \C \), the {finite-finite iterated beta integral} is defined by
	\[
	\B_\gamma^{\finfin}\left( p; \stR{x_0}{\alpha_0} \stR{x_1}{\alpha_1} \cdots \stL{x_{n-1}}{\alpha_{n-1}} \stL{x_n}{\alpha_n}  \right) \coloneqq {I_\gamma(p;  \ff{x_0,x_1}{\alpha_0,\alpha_1} \, \ff{x_1,x_2}{\alpha_1,\alpha_2} \, \cdots \ff{x_{n-1},x_n}{\alpha_{n-1},\alpha_n}  ; x_n)} \,.
	\]
	
	If \( p \neq x_0 \), we call this the \emph{incomplete} iterated beta integral.  When \( p = x_0 \), we call it the \emph{complete} iterated beta integral, and suppress the lower bound from the notation, to write
	\[
		\B_\gamma^{\finfin}\left(  \stR{x_0}{\alpha_0} \stR{x_1}{\alpha_1} \cdots \stL{x_{n-1}}{\alpha_{n-1}} \stL{x_n}{\alpha_n}  \right) 
		\coloneqq \B_\gamma^{\finfin}\left( x_0; \stR{x_0}{\alpha_0} \stR{x_1}{\alpha_1} \cdots \stL{x_{n-1}}{\alpha_{n-1}} \stL{x_n}{\alpha_n}  \right) \,.
	\]
	
	\subsubsection*{Finite-infinite iterated beta integral}
	
	When the path \( \gamma \)  and extends to \( \infty \in \mathbb{P}^1(\C) \), we define the {finite-infinite iterated beta integral} as
	\[
	\B_\gamma^{\fininf}\left( p; \stR{x_0}{\alpha_0} \stR{x_1}{\alpha_1} \cdots \stL{x_{n-1}}{\alpha_{n-1}} \stL{x_n}{\alpha_n} \right) \coloneqq {I_\gamma(p; \ff{x_0,x_1}{\alpha_0,\alpha_1} \, \ff{x_1,x_2}{\alpha_1,\alpha_2} \, \cdots \ff{x_{n-1},x_n}{\alpha_{n-1},\alpha_n}  ; \infty)} \,.
	\]
	If \( p \neq x_0 \), we call this the \emph{incomplete} iterated beta integral.  When \( p = x_0 \), we call it the \emph{complete} iterated beta integral, and write
	\[
	\B_\gamma^{\fininf}\left( \stR{x_0}{\alpha_0} \stR{x_1}{\alpha_1} \cdots \stL{x_{n-1}}{\alpha_{n-1}} \stL{x_n}{\alpha_n}  \right) 
	\coloneqq \B_\gamma^{\fininf}\left( x_0;  \stR{x_0}{\alpha_0} \stR{x_1}{\alpha_1} \cdots \stL{x_{n-1}}{\alpha_{n-1}} \stL{x_n}{\alpha_n}  \right) \,.
	\]

	\subsubsection*{Evaluation formulas} When \( n = 1 \), the complete iterated beta integrals can be evaluated via the classical beta function \( B(z_1,z_2) = \frac{\Gamma(z_1)\Gamma(z_2)}{\Gamma(z_1 + z_2)} \), as follows:
	\begin{equation}\label{eqn:betafin:neq1}
	\begin{aligned}[c]
		 \B_\gamma^{\finfin}\left(\stR{x}{\alpha} \st{y}{\beta}  \right) &=  (x-y)^{\beta-\alpha} (-1)^\alpha B(1-\alpha, \beta) \,, \\
		 \B_\gamma^{\fininf}\left(\stR{x}{\alpha} \st{y}{\beta}  \right) &=  (x-y)^{\beta-\alpha}  B(1-\alpha, \alpha-\beta) \,.
	\end{aligned}
	\end{equation}

	\subsubsection*{Residue at \texorpdfstring{\( \alpha_0 = 1\)}{alpha\_0=1}} As a function of \( \alpha_0 \), the complete iterated beta integral \( \B^{\finbul}_\gamma \), \( \bullet \in \{ \upf, \infty \} \), has a pole at \( \alpha_0 = 1 \), whose residue is given in terms of the incomplete beta integral:
	\begin{equation}\label{eqn:beta:residue}
		\res_{\alpha_0 = 1} \B_\gamma^{\finbul}\left( \stR{x_0}{\alpha_0} \stR{x_1}{\alpha_1} \cdots \stL{x_{n-1}}{\alpha_{n-1}} \stL{x_n}{\alpha_n} \right) 
		= 
		-(x_0 - x_1)^{\alpha_1 - 1} \B_\gamma^{\finbul}\left( x_0;  \stR{x_1}{\alpha_1} \cdots \stL{x_{n-1}}{\alpha_{n-1}} \stL{x_n}{\alpha_n} \right)  \,.
	\end{equation}
	
	\subsubsection*{Normalized integrals} The normalized beta integral satisfies better properties than the versions above.  We define the normalized version of the complete finite-(in)finite iterated beta integral by
		\begin{equation}\label{eqn:beta:norm}
	\widehat{\B}_\gamma^{\finbul}\left( \stR{x_0}{\alpha_0} \stR{x_1}{\alpha_1} \cdots \stL{x_{n-1}}{\alpha_{n-1}} \stL{x_n}{\alpha_n} \right) \coloneqq \frac{\B_\gamma^{\finbul}\left( \stR{x_0}{\alpha_0} \stR{x_1}{\alpha_1} \cdots \stL{x_{n-1}}{\alpha_{n-1}} \stL{x_n}{\alpha_n}  \right)}{\B_\gamma^{\finbul}\left( \stR{x_0}{\alpha_0}  \st{x_n}{\alpha_n} \right)} \,.
	\end{equation}
	
	The first useful property is that \( \widehat{\B}_\gamma^{\finbul} \), \( \bullet \in \{ \upf, \infty \} \), is holomorphic at \( \alpha_0 = 1 \).  The value is obtained via limit \( \alpha_0 \to 1 \), using the residue formula \eqref{eqn:beta:residue} and the evaluation at \( n = 1 \) from \eqref{eqn:betafin:neq1}.  In particular
	\begin{equation}\label{eqn:normbeta:a0eq1}
	 \widehat{\B}_\gamma^{\finbul}\left( \stR{x_0}{1} \stR{x_1}{\alpha_1} \cdots \stL{x_{n-1}}{\alpha_{n-1}} \stL{x_n}{\alpha_n} \right) 
	= 
	\frac{(x_0 - x_n)^{1-\alpha_n}}{(x_0 - x_1)^{1- \alpha_1}} \B_\gamma^{\finbul}\left( x_0;  \stR{x_1}{\alpha_1} \cdots \stL{x_{n-1}}{\alpha_{n-1}} \stL{x_n}{\alpha_n} \right)  \,.
	\end{equation}
	
	The most important property is that the normalized iterated beta integrals satisfy a translation invariance in the \(\alpha\)-parameters.
	
	\begin{Thm}[Hirose-Sato, \cite{HiroseSatoBeta}]\label{thm:beta:translation}
		For any \( c \in \mathbb{C} \), and \( \bullet \in \{ \upf, \infty \} \), the following translation invariance holds
		\[
			\widehat{\B}_\gamma^{\finbul}\left( \stR{x_0}{\alpha_0} \stR{x_1}{\alpha_1} \cdots \stL{x_{n-1}}{\alpha_{n-1}} \stL{x_n}{\alpha_n} \right)
			= \widehat{\B}_\gamma^{\finbul}\left( \stR{x_0}{\alpha_0 + c} \stR{x_1}{\alpha_1 + c} \cdots \stL{x_{n-1}}{\alpha_{n-1} + c} \stL{x_n}{\alpha_n + c}  \right) \,.
		\]
	\end{Thm}

    \section{Proof of the convolution formula}\label{sec:mainproof}
    
The goal of this section is to gives a proof of the convolution formula (\autoref{thm:main:conv}).  This is the missing step to establish the Main Theorem (\autoref{thm:main:desc}), from which we obtain the explicit Galois descent for M$t$V's of maximal height \( t(k_1,\ldots,k_d) \), \( k_d \geq 2 \). \medskip

Firstly, using the integral representation \eqref{eqn:zhalf:int}, we have that
\[
	\ztw^\half(\ell_1,\ldots,\ell_r) = (-1)^r I(0; 2e_1\,e_0^{\ell_1-1} \,(2e_1 - e_0)\, e_0^{\ell_2-1}\, \cdots (2e_1 - e_0)\, e_0^{\ell_r-1}; 1)
\]
Under the change of variables \( t = s^{-1} \), the end points of the path change from \( 0 \) and \( 1 \) to \( \infty \) and \( 1 \) respectively.  We also have
\begin{align*}
	& e_1(t) = \frac{\dd t}{t - 1} = \frac{\dd  s^{-1}}{s^{-1} - 1} = \frac{\dd s}{s(s-1)} = \frac{\dd s}{s-1} - \frac{\dd s}{s} = e_1(s) - e_0(s) \,, \text{ and likewise} \\
	& e_0(t) = \frac{\dd t}{t} = \frac{\dd s^{-1}}{s^{-1}} = -\frac{\dd s}{s} = e_0(s) \,.
\end{align*}
Hence
\begin{align*}
	\ztw^\half(\ell_1,\ldots,\ell_r) &= (-1)^{\sum \ell_i} I(\infty; 2(e_1-e_0)\,e_0^{\ell_1-1}\, (2e_1 - e_0)\, e_0^{\ell_2-1}\, \cdots \,(2e_1 - e_0) \,e_0^{\ell_r-1}; 1) \\
	& = I(1; e_0^{\ell_r-1}\, (2e_1 - e_0)\, e_0^{\ell_{r-1}-1} \,\cdots\, (2e_1 - e_0)\, e_0^{\ell_1-1}\, 2(e_1-e_0); \infty) \,,
\end{align*}
where the last line follows by reversal of paths.

If we now restrict to \( (\ell_1,\ldots,\ell_r) \in \mathbb{I}_{\ooe} \), i.e. \( \ell_1,\ldots,\ell_{r-1}\) are odd, and \( \ell_r \) is even, we can write
\begin{equation}\label{eq:zeta_half_int}
	\ztw^\half(\ell_1,\ldots,\ell_r) =
	I(1;f_{1,0}\,(f_{0,1}\,f_{1,0})^{\frac{\ell_{r}-2}{2}}\,f_{0,0}\,\cdots\,(f_{0,1}\,f_{1,0})^{\frac{\ell_{2}-1}{2}}f_{0,0}\,(f_{0,1}\,f_{1,0})^{\frac{\ell_{1}-1}{2}}\,(f_{0,0}-f_{0,1});\infty) \,,	 
\end{equation}
where
\[
f_{a,b}(t)\coloneqq\begin{cases}
2e_{1}(t)-e_{0}(t) & (a,b)=(0,0)\\
2e_{-1}(t)-e_{0}(t) & (a,b)=(1,1)\\
e_{0}(t) & \text{$(a,b)=(0,1)$ or $(1,0)$}\,.
\end{cases}
\]

We now have the following proposition, to rewrite this expression.

\begin{Prop}
\label{prop:Iterated_integral_relation_half}Let $k\geq2$, and set  \( \omega(t) = \frac{\dd t}{\sqrt{t} \, (t-1)} \).  For a
general sequence $a_{1},a_{2},\ldots,a_{k},a_{k}'\in\{0,1\}$, we
have
\begin{align*}
& I(1;f_{1,a_{1}}\,f_{a_{1},a_{2}}\,f_{a_{2},a_{3}}\,\cdots\, f_{a_{k-2},a_{k-1}}\,(f_{a_{k-1},a_{k}}-f_{a_{k-1},a_{k}'});\infty) \\
&= -\frac{1}{\ii \pi}I(0;\omega \, e_{a_{1}} \, e_{a_{2}} \, \cdots \, e_{a_{k-1}} \, (e_{a_{k}}-e_{a_{k}'});-\infty)\,.
\end{align*}
\end{Prop}

\begin{proof}
	Firstly, under the change of variables \( u = \big( \frac{\sqrt{t} - 1/\sqrt{t}}{2\ii} \big)^2 = -\frac{(t-1)^2}{4t} = -\frac{t - 2 + t^{-1}}{4} \), the forms \( f_{a,b}(t) \) transform as follows,
	\begin{align*}
		& f_{0,0}(t) = \frac{\dd u}{u} = \ff{0,0}{\half,\half}(u) \,, \quad f_{1,1}(t) = \frac{\dd u}{u-1} = \ff{1,1}{\half,\half}(u) \,, \\
		& f_{0,1}(t) = f_{1,0}(t) = \frac{\dd u}{\sqrt{u(u-1)}} = \ff{0,1}{\half,\half}(u) = \ff{1,0}{\half,\half}(u) \,.
	\end{align*}
	Moreover, the integration path from 1 to \( \infty \) along the positive real axis in \( t \)-coordinates is transformed into the straight line path from 0 to \( -\infty \) along the the negative real axis in \( u \)-coordinates.
	
	Therefore, we can directly write
	\begin{align*}
 & I(1;f_{1,a_{1}}\,f_{a_{1},a_{2}}\,f_{a_{2},a_{3}}\,\cdots \,f_{a_{k-2},a_{k-1}}\,(f_{a_{k-1},a_{k}}-f_{a_{k-1},a_{k}'});\infty)\\
& =\lim_{\alpha \to \half^+}
\left( 
\B^\fininf\left(0; \stR{1}{\half} \stR{a_1}{\half} \cdots \stL{a_{k-1}}{\half} \stL{a_k}{1-\alpha} \right)
-
\B^\fininf\left(0; \stR{1}{\half} \stR{a_1}{\half} \cdots \stL{a_{k-1}}{\half} \stL{a_k'}{1-\alpha} \right)
\right) \,,
\end{align*}
and likewise
	\begin{align*}
& I(0; \omega \, e_{a_1} \, e_{a_2} \, \cdots \, e_{a_{k-1}} \, (e_{a_{k}} - e_{a_{k}'}) ; -\infty)\\
& =\lim_{\alpha \to 0^-}
\left( 
\B^\fininf\left(\stR{0}{\half} \stR{1}{0} \stR{a_1}{0} \cdots \stL{a_{k-1}}{0} \stL{a_k}{\alpha} \right)
-
\B^\fininf\left(\stR{0}{\half} \stR{1}{0} \stR{a_1}{0} \cdots \stL{a_{k-1}}{0} \stL{a_k'}{\alpha} \right)
\right) \,.
\end{align*}
Using the theory of iterated beta integrals, we shall relate these two expressions to establish the proposition.  By the formula \eqref{eqn:normbeta:a0eq1} for the value of \( \widehat{\B}^\fininf \), when \( \alpha_0 = 1 \), we have 
\begin{align*}
 & \frac{(-z_{n})^{\alpha}}{\ii\sqrt{z_{1}}}\B^\fininf\left( 0; \stR{z_1}{\half} \stR{z_2}{\half}  \cdots \stL{z_{n-1}}{\half} \stL{z_n}{1-\alpha} \right) \\
 & =\widehat{\B}^\fininf\left( \stR{0}{1} \stR{z_1}{\half} \stR{z_2}{\half}  \cdots \stL{z_{n-1}}{\half} \stL{z_n}{1-\alpha} \right)  \,.
\end{align*}
Applying translation invariance in \autoref{thm:beta:translation} (with \( c = -\half \)) gives
\[
 =\widehat{\B}^\fininf\left( \stR{0}{\half} \stR{z_1}{0} \stR{z_2}{0}  \cdots \stL{z_{n-1}}{0} \stL{z_n}{\half-\alpha} \right) \,.
\]
Expressing this in terms of \( \B^\fininf \) by definition of the normalized integral \eqref{eqn:beta:norm}, and applying the evaluation formula \eqref{eqn:betafin:neq1} for \( \B^\fininf_\gamma\left( \stR{x}{\alpha} \st{y}{\beta} \right)\) gives
\begin{align*}
& = \frac{\B^\fininf\left(\stR{0}{\half}  \stR{z_1}{0} \stR{z_2}{0}  \cdots \stL{z_{n-1}}{0} \stL{z_n}{\half-\alpha} \right)}{\B^\fininf\left( \stR{0}{\half} \st{z_n}{\half-\alpha} \right)} \\
& = \frac{(-z_n)^\alpha}{\B(\half, \alpha)}  \B^\fininf\left(\stR{0}{\half} \stR{z_1}{0} \cdots \stL{z_{n-1}}{0} \stL{z_n}{\half-\alpha} \right) \,,
\end{align*}
where here \( \B(\alpha,\beta) = \frac{\Gamma(\alpha)\Gamma(\beta)}{\Gamma(\alpha+\beta)} \) is the usual beta function.  After cancelling the factor \( (-z_1)^\alpha \) from both sides of this overall equality, we obtain
\begin{align*}
\B^\fininf\left(  0; \stR{z_1}{\half} \stR{z_2}{\half}  \cdots \stL{z_{n-1}}{\half} \stL{z_n}{1-\alpha} \right) 
 = -\frac{\sqrt{z_{1}}}{\ii\B(\half, \alpha)}  \B^\fininf\left(\stR{0}{\half}  \stR{z_1}{0} \cdots \stL{z_{n-1}}{0} \stL{z_n}{\half-\alpha} \right) \,,
\end{align*}
Notice that both sides are convergent for \( \Re(\alpha) > \half \).  Setting \( (z_1,z_2,\ldots.z_n) = (1, a_1, a_2,\ldots, a_{k-1}, b) \), with \( b \in \{ a_k, a_k' \} \), and taking the difference of the two resulting equalities gives the identity
\begin{align*}
	& \B^\fininf\left( 0; \stR{1}{\half} \stR{\alpha_1}{\half} \stR{\alpha_2}{\half}  \cdots \stL{\alpha_{k-1}}{\half} \stL{\alpha_k}{1-\alpha} \right) 
	- \B^\fininf\left( 0; \stR{1}{\half} \stR{\alpha_1}{\half} \stR{\alpha_2}{\half}  \cdots \stL{\alpha_{k-1}}{\half} \stL{\alpha_k'}{1-\alpha}  \right)  \\
	=
	& -\frac{1}{\ii\B(\half, \alpha)} \left( \B^\fininf\left( \stR{0}{\half} \stR{1}{0} \stR{a_1}{0} \cdots \stL{a_{k-1}}{0} \stL{a_k}{\half-\alpha} \right)
	-
	\B^\fininf\left( \stR{0}{\half} \stR{1}{0} \stR{a_1}{0} \cdots \stL{a_{k-1}}{0} \stL{a_k'}{\half-\alpha} \right) \right) \,.
\end{align*}
Letting \( \alpha \to\half^+ \), we have \( \B(\half,\alpha) \to \B(\half,\half) = \pi \), and hence obtain the desired result.
\end{proof}

We then express the integral over \( e_{a_i} \) as a convolution of multiple $t$-values and multiple zeta star values.

\begin{Prop}\label{prop:iterated integral expression for convoluted sum}
	For \( d \geq 1 \), and \( k_1, \ldots, k_d \geq 2 \), we have
	\[
		-\frac{1}{\ii \pi} I(0; \omega \, e_0^{k_1-1} \, e_1 \, e_0^{k_2-1} \, \cdots \, e_1 \, e_0^{k_d-1} \, (e_1 - e_0); -\infty) = \sum_{i=0}^d (-1)^i \ttw(k_1,\ldots,k_i) \zeta^\star(k_d,\ldots,k_{i+1}) \,.
	\]	
	Here \( \omega(t) = \frac{\dd t}{\sqrt{t} \, (t-1)} \), and the branch of \( \sqrt{t} \) is chosen to be in \( \ii \mathbb{R}_{>0} \) for \( t \in \mathbb{R}_{<0} \), i.e. along the integration path.
\end{Prop}

\begin{proof}
Let $W\coloneqq\omega e_{0}^{k_{1}-1}e_{1}e_{0}^{k_{2}-1}\cdots e_{1}e_{0}^{k_{d}-1}(e_{1}-e_{0})$, then $I\left(0;W;-\infty\right)$ be the iterated integral on the left-side of the equality.  Under the change of variable $u=\sqrt{t}$,
the differential forms transform into
\begin{align*}
e_{0}(t) & = \frac{\dd t}{t} = \frac{\dd u^2}{u^2} = \frac{2u \dd u}{u^2} = 2e_{0}(u) \,, \\
e_{1}(t) & =e_{0}(t)\frac{t}{t-1}=\frac{2u\dd u}{u^{2}-1}=e_{1}(u)+e_{-1}(u) \,, \\
\omega(t) & =e_{0}(t)\frac{\sqrt{t}}{t-1}=\frac{2\dd u}{u^{2}-1}=e_{1}(u)-e_{-1}(u) \,.
\end{align*}
Likewise, the straight line integration path from \( 0 \) to \( -\infty \) is transformed into the straight line path from \( 0 \) to \( \ii\infty \) along the positive imaginary axis.  We thus have the equality
\begin{align*}
& I(0;W;-\infty) \\
& = I(0;(e_{1}-e_{-1})\,(2e_{0})^{k_{1}-1}\,(e_{1}+e_{-1})\,(2e_{0})^{k_{2}-1}\,\cdots\,(e_{1}+e_{-1})\,(2e_{0})^{k_{d}-1}\,(e_{1}+e_{-1}-2e_{0});\ii\infty) 
\end{align*}
By homotopy of paths (the first form has no pole at \( 0 \), and the final form has no pole at \( \ii \infty = +\infty \in \mathbb{P}^1(\mathbb{C}) \), so the integral is independent of tangent vectors), this equals
\begin{align*}
& = I_\gamma(0;(e_{1}-e_{-1})\,\underbrace{(2e_{0})^{k_{1}-1}\,(e_{1}+e_{-1})\,(2e_{0})^{k_{2}-1}\,\cdots\,(e_{1}+e_{-1})\,(2e_{0})^{k_{d}-1}\,(e_{1}+e_{-1}-2e_{0})}_{\eqqcolon W'}; +\infty) \,,
\end{align*}
where \( \gamma \) is the path from \( 0 \) to \( +\infty \) along the positive real axis, indented upwards at 1 to avoid the pole.

We claim now that the integral is purely imaginary.  Notice the path \( -\gamma \) is homotopic to the path \( \overline{\gamma} \).  Write \( W' \), as indicated above, for the forms appearing after position one in the integrand; under the change of variables \( u \mapsto -u \), we then have
\begin{align*}
	I_\gamma(0; (e_1-e_{-1}) \, W'; +\infty) 
	&= 
	I_{-\gamma}(0; (e_{-1}-e_1)\, W'; -\infty) \\
	&= 
	I_{\overline{\gamma}}(0; (e_{-1}-e_1)\, W'; +\infty) \\
	&= 
	-\overline{ I_{\gamma}(0; (e_1-e_{-1})\, W'; +\infty) } \,.
\end{align*}
This establishes that the integral is purely imaginary:
\begin{equation}\label{eq:purely-imaginary}
	I(0; W; -\infty) = I_\gamma(0; (e_1-e_{-1}) \, W'; +\infty) \in \ii\mathbb{R} \,.
\end{equation}
Returning to the original variable \( t = u^2 \), we note that the path \( \gamma \) in the \( u \)-plane is transformed to a path homotopic to \( \gamma \) in the \( t \)-plane.  Thus we have
\[
	I(0; W; -\infty) = I_\gamma(0; W, +\infty) \,.
\]

Now let $1_{\pm}$ denote the tangential basepoints at
$1$ with tangential vector $\pm1$.  We shall decompose the path \( \gamma \) into three sections: the straight line from 0 to \( 1_- \), a small negatively-oriented loop upwards from \( 1_- \) to \( 1_{+} \), and the straight line path from \( 1_+ \) to \( + \infty \).  By the decomposition of paths formula
\[
	I_\gamma(0; W; +\infty) = \sum_{W_1 W_2 W_3 = W} I(0; W_1; 1_-) \, I_C(1_-; W_2; 1_+) I(1+; W_3; +\infty) \,,
\]
where the sum is taken over all ways of deconcatenating \( W \) into three subwords.

By the assumption that \( k_1,\ldots,k_d \geq 2 \), any subword of length \( \geq 2 \) contains the form \( e_0 \), which has no pole at \( 1 \).  In particular, if \( W_2 \) has length \( \geq 2 \), then \( W_2 \) ends in \( e_0e_1^n \), for some \( 0 \leq n \leq 1 \).  By shuffle regularization, such a word \( W_2 \) is a polynomial of degree \( n \) in \( e_1 \), whose coefficients are words (of complementary length) ending in \( e_0 \).  But any integral of the form \( I_C(1_-; \cdots\,e_0; 1_+) \) vanishes as \( I_C(1_-; \cdots\,e_0; 1_+)= I_C(1_-; \cdots\,e_0; 1_-) = 0 \) by the homotopy rotating the final tangent vector \( 1_+ \) to \( 1_- \).  Hence \( I_C(1_-; W_2; 1_+) = 0 \), whenever \( W_2 \) has length \( \geq 2 \).

From this, we only need to consider summands where \( W_2 = \mathbbm{1}, \omega, e_1 \) or \( e_1 - e_0 \).  In other words,
\[
	I_\gamma(0; W; +\infty) = S_0 + S_1 \,,
\]
where
\begin{align*}
\arraycolsep=0pt%
\renewcommand{\arraystretch}{1.5}
\begin{array}{rcl}
	S_0 = {} & \displaystyle \sum_{W_1 W_3 = W} & I(0; W_1; 1_-) I(1_+; W_3; \infty) \\[\baselineskip]
	S_1 = {} & \displaystyle \sum_{\substack{W_1 W_2 W_3 = W \\ W_2 \in \{ \omega, e_1, e_1-e_0\}}} & I(0; W_1; 1_-) I(1_-; W_2; 1_+) I(1_+; W_3; \infty) \,.
	\end{array}
\end{align*}
Note that
\[
\omega(t)=e_{1}(t)-\overbrace{\frac{dt}{\sqrt{t}(1+\sqrt{t})}}^{\mathclap{\text{holomorphic at $t=1$}}} \,.
\]
It therefore follows that, for \( W_2 \in \{ \omega, e_1, e_1-e_0 \} \), one has
\[
	I_C(1_-; W_2; 1+) = I_C(1_-; e_1; 1_+) = -\ii\pi \,.
\]
Thus
\[
	S_1 =  -\ii \pi \sum_{\substack{W_1 W_2 W_3 = W \\ W_2 \in \{ \omega, e_1, e_1-e_0\}}} I(0; W_1; 1_-)  I(1_+; W_3; \infty) \,.
\]
We see now that \( S_0 \in \mathbb{R} \), while \( S_1 \in \ii\mathbb{R} \), which -- together with \eqref{eq:purely-imaginary} -- implies that \( S_0 = 0 \).

Finally, we express \( S_1 \) in terms of multiple $t$-values and multiple zeta star values.  We have
\begin{align*}
	-\frac{1}{\ii \pi} S_1 = {} & \sum_{\substack{W_1 W_2 W_3 = W \\ W_2 \in \{ \omega, e_1, e_1-e_0\}}} I(0; W_1; 1_-)  I(1_+; W_3; \infty) \\
	= {}  & I(1_{+};e_{0}^{k_{1}-1}\,e_{1}\,e_{0}^{k_{2}-1}\,\cdots\, e_{1}\,e_{0}^{k_{d}-1}\,(e_{1}-e_{0});+\infty)\\
	& {} + \sum_{i=1}^{d-1}I(0; \omega \, e_{0}^{k_{1}-1} \, e_{1 \, }e_{0}^{k_{2}-1} \, \cdots e_{1} \, e_{0}^{k_{i}-1};1_{-}) \, I(1_{+};e_{0}^{k_{i+1}-1}\,e_{1}\,e_{0}^{k_{2}-1}\,\cdots\, e_{1}\,e_{0}^{k_{d}-1}\,(e_{1}-e_{0});+\infty) \\
	& +I(0;\omega\, e_{0}^{k_{1}-1}\, e_{1}e \, _{0}^{k_{2}-1}\,\cdots \,e_{1} \, e_{0}^{k_{d}-1};1_{-}) \,,
\end{align*}
where the integrals no longer depend on tangent vectors.
From the integral representation \eqref{eqn:int:ttw}, we have
\[
	I(0; \omega \, e_{0}^{k_{1}-1} \, e_{1 \, }e_{0}^{k_{2}-1} \, \cdots e_{1} \, e_{0}^{k_{i}-1};1) = (-1)^i \ttw(k_1,k_2,\ldots,k_i) \,.
\]
After reversing the path, and applying the change of variables \( t \mapsto t^{-1}  \), we also have
\begin{align*}
	& I(1;e_{0}^{k_{i+1}-1}\,e_{1}\,e_{0}^{k_{2}-1}\,\cdots\, e_{1}\,e_{0}^{k_{d}-1}\,(e_{1}-e_{0});+\infty) \\
	&=(-1)^{\sum_{j=i+1}^{d}k_{j}}I(+\infty;(e_{1}-e_{0})\,e_{0}^{k_{d}-1}\,e_{1}\,e_{0}^{k_{d-1}-1}\,\cdots\, e_{1}\,e_{0}^{k_{i+1}-1};1)\\
	& =I(0;(-e_{1})\,e_{0}^{k_{d}-1}\,(e_{0}-e_{1})\,e_{0}^{k_{d-1}-1}\,\cdots\,(e_{0}-e_{1})\,e_{0}^{k_{i+1}-1};1)\\
	& =\zeta^{\star}(k_{d},k_{d-1},\ldots,k_{i+1}) \,,
\end{align*}
via the iterated integral representation from \eqref{eqn:zstar:int}.  It readily follows that
\[
	-\frac{1}{\ii \pi} I (0;W;-\infty)=-\frac{1}{\ii\pi } I_{\gamma}(0;W;+\infty) = \sum_{i=0}^{d} (-1)^{i} \ttw(k_{1},k_{2},\ldots,k_{i})\zeta^{\star}(k_{d},k_{d-1},\ldots,k_{i+1}) \,.
\]
This finishes the proof of the proposition.
\end{proof}
  
Finally, we can complete the proof of \autoref{thm:main:conv}.  From  \eqref{eq:zeta_half_int} and Proposition \ref{prop:Iterated_integral_relation_half}, we have
that
\[
-\ztw^\half(\bl(k_{d},\ldots,k_{1})) = -\frac{1}{\ii \pi} I(0;\omega\, e_{0}^{k_{1}-1}\,e_{1}\,e_{0}^{k_{2}-1}\,\cdots\, e_{1}\,e_{0}^{k_{d}-1}\,(e_{1}-e_{0});-\infty)
\]
when $k_{1},\ldots,k_{d}\geq2$. By \autoref{prop:iterated integral expression for convoluted sum}, this is
\[
=
\sum_{i=0}^{d}(-1)^{i} \ttw(k_{1},\ldots,k_{i})\zeta^{\star}(k_{d},\ldots,k_{i+1}) \,.
\]
So the claimed convolution identity holds, and \autoref{thm:main:conv} is proven. \hfill \qedsymbol

\bibliographystyle{habbrv2}
\bibliography{bibliography}

\end{document}